%% file: mainArxiv.tex
\documentclass[11pt,reqno,a4paper]{amsart}
\usepackage[utf8]{inputenc}
 \usepackage[doi=true,isbn=false,style=alphabetic,sorting=nyt,backend=biber,maxnames=10,maxalphanames=3,giveninits=true]{biblatex}
\AtEveryBibitem{\clearlist{language}}
\bibliography{main}

\usepackage[breaklinks,colorlinks,plainpages,hypertexnames=false,plainpages=false]{hyperref}
\hypersetup{urlcolor=blue, citecolor=blue, linkcolor=blue}

\usepackage{amsmath}
\usepackage{amsthm}
\usepackage{thmtools} 
\usepackage{amssymb}
\usepackage{amsfonts}
\usepackage{bbm} 
\usepackage{mathtools}
\usepackage{stmaryrd}
\usepackage{enumitem}
\usepackage{mathrsfs}
\usepackage{comment}
\usepackage{listings}
\usepackage[capitalise, noabbrev]{cleveref} 
\usepackage[noend]{algorithmic} 

\usepackage{mathdots}

\usepackage{jlcode}

\usepackage{array}
\newcolumntype{L}[1]{>{\raggedright\let\newline\\\arraybackslash\hspace{0pt}}m{#1}}
\newcolumntype{C}[1]{>{\centering\let\newline\\\arraybackslash\hspace{0pt}}m{#1}}
\newcolumntype{R}[1]{>{\raggedleft\let\newline\\\arraybackslash\hspace{0pt}}m{#1}}

\usepackage{tikz}
\usetikzlibrary{arrows}
\usetikzlibrary{decorations.pathreplacing}
\usetikzlibrary{matrix}
\usetikzlibrary{calc}
\usetikzlibrary{shapes}
\usetikzlibrary{patterns}
\usetikzlibrary{fit,backgrounds,scopes}
\usepackage[dvipsnames]{xcolor}

\newtheoremstyle{theoremstyle}
{10pt}      
{5pt}       
{\itshape}  
{}          
{\bfseries} 
{}         
{ }      
{}          

\newtheoremstyle{algorithmstyle}
{10pt}      
{5pt}       
{}  
{}          
{\bfseries} 
{}         
{ }      
{}          

\newtheoremstyle{examplestyle}
{10pt}      
{5pt}       
{}          
{}          
{\bfseries} 
{}         
{ }      
{}          

\makeatletter 
\newcommand{\subalign}[1]{%
  \vcenter{%
    \Let@ \restore@math@cr \default@tag
    \baselineskip\fontdimen10 \scriptfont\tw@
    \advance\baselineskip\fontdimen12 \scriptfont\tw@
    \lineskip\thr@@\fontdimen8 \scriptfont\thr@@
    \lineskiplimit\lineskip
    \ialign{\hfil$\m@th\scriptstyle##$&$\m@th\scriptstyle{}##$\hfil\crcr
      #1\crcr
    }%
  }%
}
\makeatother

\theoremstyle{theoremstyle}
\newtheorem{theorem}{Theorem}[section]
\newtheorem{lemma}[theorem]{Lemma}
\newtheorem{proposition}[theorem]{Proposition}
\newtheorem{corollary}[theorem]{Corollary}

\newtheorem*{theorem*}{Main theorem}
\theoremstyle{examplestyle}
\newtheorem{example}[theorem]{Example}
\newtheorem{definition}[theorem]{Definition}

\newtheorem{notation}[theorem]{Notation}
\theoremstyle{algorithmstyle}

\renewcommand{\aa}{\mathbbm{a}}
\newcommand{\bb}{\mathbbm{b}}
\newcommand{\cc}{\mathbbm{c}}

\newcommand{\EE}{\mathbb{E}}
\newcommand{\ee}{\mathbbm{e}}
\newcommand{\FF}{\mathbb{F}}
\newcommand{\GG}{\mathbb{G}}
\newcommand{\HH}{\mathbb{H}}

\newcommand{\RR}{\mathbb{R}}

\newcommand{\ZZ}{\mathbb{Z}}
\newcommand{\suchthat}{\;\ifnum\currentgrouptype=16 \middle\fi|\;}
\newcommand{\bigmid}{\left.\vphantom{\Big\{} \suchthat \vphantom{\Big\}}\right.}

\newcommand{\bbone}{\mathbbm{1}}

\newcommand*\circled[1]{\tikz[baseline=(char.base)]{
            \node[shape=circle,draw,inner sep=2pt] (char) {#1};}}
            
\DeclareMathOperator{\codim}{codim}

\DeclareMathOperator{\mult}{mult}
\DeclareMathOperator{\Relint}{Relint}
\DeclareMathOperator{\Span}{Span}
\DeclareMathOperator{\Trop}{Trop}
\DeclareMathOperator{\rank}{rank}
\DeclareMathOperator{\Star}{Star}

\DeclareMathAlphabet{\altmathbb}{U}{fplmbb}{m}{n} 

\title{The tropical galaxy of a Laman graph}
\author{Amelia Bielby}
\address{Amelia Bielby, Department of Mathematical Sciences, Durham University.}
\email{amelia.bielby@durham.ac.uk}
\author{Arushi Chauhan}
\address{Arushi Chauhan, Department of Mathematical Sciences, Durham University.}
\email{arushi.chauhan@durham.ac.uk}
\author{Cassia Pearce}
\address{Cassia Pearce, Department of Mathematical Sciences, Durham University.}
\email{cassia.pearce@durham.ac.uk}
\author{Yue Ren}
\address{Yue Ren, Department of Mathematical Sciences, Durham University.}
\email{yue.ren2@durham.ac.uk}
\urladdr{https://www.yueren.de/}
\date{\today}

\newcommand{\pointOrNoPoint}{.}
\newcommand{\forGlasgow}[1]{}
\newcommand{\forArXiv}[1]{#1}

\begin{document}

\begin{abstract}
A Laman graph $G$ is a minimally rigid graph in dimension two, and its realization number is its number of distinct embeddings with fixed generic edge lengths. While conjectured to grow exponentially in the number of vertices of $G$, the best proven lower bound is merely $2$. Motivated by the fact that the realization number can be expressed as a tropical intersection product involving $\Trop(G)$, the Bergman fan of the graphic matroid of $G$, and the fact that stars of $\Trop(G)$ naturally lead to lower bounds thereof, we introduce the tropical galaxy of $G$ together with a galactic pairing thereon. We study structural properties of this pairing, such as under which conditions it is non-trivially subadditive, and connect it being non-zero to arboreal pairs.  We also present a software package for working with tropical galaxies.
\end{abstract}

\maketitle



\input{introduction}

\subsection*{Acknowledgements}
We would like to thank Oliver Clarke (Durham) and Ben Smith (Lancaster) for helpful discussions.  Yue Ren is supported by the UKRI Future Leaders Fellowship ``Computational tropical geometry and its applications'' (MR/Y003888/1) as well as the EPSRC grant ``Mathematical Foundations of Intelligence: An "Erlangen Programme" for AI'' (EP/Y028872/1).

\input{background}

\input{excision}

\input{galaxy}

\input{subadditivity}

\input{arborealPairs}

\input{software}

\renewcommand*{\bibfont}{\small}
\printbibliography
\end{document}

%% file: introduction.tex
\section{Introduction}
Rigidity theory studies the (in)flexibility of graphs embedded in euclidean space.  Its origins date back to Maxwell, who studied bar-and-joint frameworks motivated by engineering \cite{Maxwell1870}.  Modern rigidity theory enjoys a variety of applications beyond structural engineering, such as robotics \cite{ZelazoFranchiAllgowerBulthoffGiordano2012}, material science \cite{RaderBrown2011}, and sensor networks \cite{SoYe2007}. Moreover, rigidity theory draws from a surprising number of mathematical areas such as algebraic geometry and combinatorics \cite{SitharamJohnSidman2018}.

Laman graphs $G$ play an important role in rigidity theory. They represent minimally rigid structures in dimension $2$.  A property of particular interest for this paper are their \emph{realization numbers} $c_2(G)$, that is the number of ways a $G$ with fixed generic edge lengths can be embedded into $\RR^2$ up to translation and rotation.  It is a property that connects combinatorics and geometry.  Most peculiarly, the realization number of a Laman graph with $n$ vertices has been conjectured to be $2^{n-3}$ by Jackson and Owen \cite{JacksonOwen2019}, while the current best lower bound is merely $2$. 

We will study the realization number using a new tropical approach that expresses it as a tropical intersection product $2\cdot c_2(G)=\Trop(G)\cdot (-\Trop(G))$, where $\Trop(G)$ is (a coarsening of) the Bergman fan of the graphic matroid of $G$ \cite{ClarkeDewarTrippMaxwellNixonRenSmith2025}.  This approach naturally leads to a method of obtaining lower bounds by replacing the intersects with their stars, see \cref{lem:intersectionProductAndStar}.  This work aims to start an avenue of research into lower bounding $c_2(G)$ by studying the stars of $\Trop(G)$.

In \cref{sec:excision}, we introduce so-called \emph{graph excisions}, which is the graphic analogue of taking tropical stars along rays.  These will be crucial building-blocks of the next sections.

In \cref{sec:galaxy}, we introduce the so-called \emph{tropical galaxy} $\Gamma_G$ of a Laman graph $G$, which is a directed acyclic graph whose vertices are the stars of $\Trop(G)$ and edges $(\Sigma_1,\Sigma_2)$ encode that $\Sigma_2$ is the star of $\Sigma_1$ around a ray.  Note that $\Gamma_G$ has a unique source, which is $G$ itself, and many sinks, which are stars around maximal cones. On $\Gamma_G$ we have a \emph{galactic pairing} given by $\langle \Sigma_1,\Sigma_2\rangle\coloneqq \Sigma_1\cdot(-\Sigma_2)$.  As $2\cdot c_2(G)=\langle \Trop(G),\Trop(G)\rangle$, the rest of the paper is dedicated to study its properties:

In \cref{sec:subadditivity}, we study conditions under which the galactic pairing is non-trivially subadditive.

In \cref{sec:arboreal}, we show that the galactic pairing is non-zero on the leafs of $\Gamma_G$ if and only if the pair of leaves is a so-called arboreal pair as in \cite{Ardila-MantillaEurPenaguiao2024}.

In \cref{sec:software}, we describe a software package that we have written in order to facilitate our experiments, which is publicly available under
\begin{center}
    \url{https://github.com/YueRen/TropicalGalaxy.jl}.
\end{center}


%% file: background.tex
\section{Background}\label{sec:background}
In this section, we briefly recall some basic concepts that are of immediate interest to us with the main purpose of fixing our notation.

\begin{notation}
  \label{not:graphs}
  Throughout the paper, we will use $[m]$ to denote the set $\{1,\dots,m\}$.
  In this paper, we consider two types of undirected and loopless graph:

  \begin{enumerate}[leftmargin=*]
  \item Simple graphs $G$ with vertex set $V(G)$ and edge set $E(G)$.  We will often use $n$ and $m$ to denote the number of vertices and edges, respectively.  We will fix an ordering $V(G)=\{v_1, \dots, v_n\}$ and $E(G)=\{e_1, \dots, e_m\}$.  In most instances, we are only interested in the edges and write $E(G)=\{1, \dots, m\}$.

  \item Loopless multigraphs $\GG$ with vertex set $V(\GG)$ and edge set $E(\GG)$.  As above, we will often use $n$ and $m$ to denote the number of vertices and edges, respectively, and we will again fix an ordering $V(G)=\{v_1, \dots, v_n\}$ and $E(G)=\{e_1, \dots, e_m\}$.  In most instances, we are only interested in the edges and again write $E(G)=\{1, \dots, m\}$.
  Moreover, we use $\EE(\GG)\subseteq 2^{E(G)}$ to denote the set of multiedges, i.e., every $\ee\in \EE(\GG)$ is a subset $\ee\subseteq E(G)$ that is the edge set of an induced subgraph with two vertices and at least one edge.
  \end{enumerate}
\end{notation}

\subsection{Laman graphs}
Laman graphs are the starting point of our studies.  In this section, we briefly recall their definition and their construction.  More information about Laman graphs and their role in rigidity theory can be found in \cite{SitharamJohnSidman2018}.

\begin{definition}
  \label{def:LamanGraphs}
  A \emph{Laman graph} is a simple graph $G$ such that
  \begin{enumerate}
  \item $|E(G)|=2\cdot |V(G)|-3$,
  \item $|E(G')|\leq 2\cdot |V(G')|-3$ for every vertex-induced subgraph $G'\subseteq G$.
  \end{enumerate}
\end{definition}

\begin{example}
  \label{ex:LamanGraphs}
  \cref{fig:LamanGraphs} shows three examples of Laman graphs with 9 edges: the prism graph, the triangle wheel with four triangles, and the triangle chain with four triangles.
\end{example}

\begin{figure}[t]
  \centering
  \begin{tikzpicture}
    \node (prism)
    {
      \begin{tikzpicture}[x={(0.7,0)},y={(0,0.7)}]
        \coordinate (v1) at (0,0);
        \coordinate (v2) at (1,1);
        \coordinate (v3) at (0,2);
        \coordinate (v4) at (4,0);
        \coordinate (v5) at (3,1);
        \coordinate (v6) at (4,2);
        \draw
        (v1) -- (v2) -- (v3) -- (v1)
        (v4) -- (v5) -- (v6) -- (v4)
        (v1) -- (v4)
        (v2) -- (v5)
        (v3) -- (v6);
        \fill
        (v1) circle (2pt)
        (v2) circle (2pt)
        (v3) circle (2pt)
        (v4) circle (2pt)
        (v5) circle (2pt)
        (v6) circle (2pt);
      \end{tikzpicture}
    };
    \node[anchor=west,xshift=5mm] (wheel) at (prism.east)
    {
      \begin{tikzpicture}[x={(0.7,0)},y={(0,0.7)}]
        \coordinate (v1) at (0,0);
        \coordinate (v2) at (0,2);
        \coordinate (v3) at (2,0);
        \coordinate (v4) at (2,2);
        \coordinate (v5) at (4,0);
        \coordinate (v6) at (4,2);
        \draw
        (v1) -- (v3) -- (v5) -- (v6) -- (v4) -- (v2) -- (v1)
        (v2) -- (v3)
        (v4) -- (v3)
        (v6) -- (v3);
        \fill
        (v1) circle (2pt)
        (v2) circle (2pt)
        (v3) circle (2pt)
        (v4) circle (2pt)
        (v5) circle (2pt)
        (v6) circle (2pt);
      \end{tikzpicture}
    };
    \node[anchor=west,xshift=5mm] (chain) at (wheel.east)
    {
      \begin{tikzpicture}[x={(0.7,0)},y={(0,0.7)}]
        \coordinate (v1) at (0,0);
        \coordinate (v2) at (0,2);
        \coordinate (v3) at (2,0);
        \coordinate (v4) at (2,2);
        \coordinate (v5) at (4,0);
        \coordinate (v6) at (4,2);
        \draw
        (v1) -- (v2) -- (v3) -- (v1)
        (v2) -- (v4) -- (v3)
        (v3) -- (v5) -- (v4)
        (v4) -- (v6) -- (v5);
        \fill
        (v1) circle (2pt)
        (v2) circle (2pt)
        (v3) circle (2pt)
        (v4) circle (2pt)
        (v5) circle (2pt)
        (v6) circle (2pt);
      \end{tikzpicture}
    };
  \end{tikzpicture}
  \caption{Three Laman graphs with $9$ edges\pointOrNoPoint}
  \label{fig:LamanGraphs}
\end{figure}
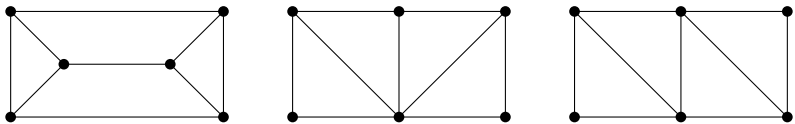

\begin{definition}
  \label{def:hennebergMoves}
  Let $G$ be a graph with vertex set $V(G)=\{v_1,\dots,v_n\}$ and edge set $E(G)\subseteq\binom{V(G)}{2}$.

\begin{description}
    \item[$(H_0)$] A \emph{Henneberg-0 move} adds a new vertex $v_{n+1}$ and two new edges $\{v_{i},v_{n+1}\}$, $\{v_{j},v_{n+1}\}$ for $v_i,v_j\in V(G)$, $i\neq j$.
    \item[$(H_1)$] A \emph{Henneberg-1 move} removes an edge $\{v_i,v_j\}\in E(G)$, adds a new vertex $v_{n+1}$ and three new edges: $\{v_{i},v_{n+1}\}$, $\{v_{j},v_{n+1}\}$, and $\{v_{k},v_{n+1}\}$ for some $i\neq k\neq j$.
\end{description}
\end{definition}

\begin{theorem}[{\cite{Henneberg1911}}]
  \label{thm:henneberg}
  All Laman graphs arise from a single edge by performing successive Henneberg moves.
\end{theorem}

\begin{example}
  \label{ex:HennebergMoves}
  \cref{fig:HennebergMoves} shows how the prism graph from \cref{ex:LamanGraphs} can be constructed from an edge via sequence of three Henneberg-$0$ moves and one Henneberg-$1$ move.
\end{example}

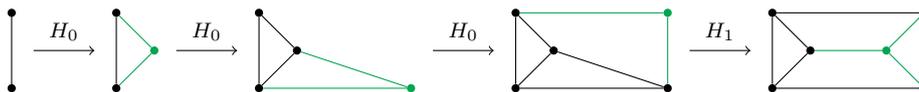
\begin{figure}[t]
  \centering
  \begin{tikzpicture}
    \node (g1)
    {
      \begin{tikzpicture}[x={(0.5,0)},y={(0,0.5)}]
        \node [circle, scale = 0.3, fill](v1) at (0,0){};
        \node [circle, scale = 0.3, fill](v2) at (0,2){};

        \draw(0,0) -- (0,2);
      \end{tikzpicture}
    };
    \node[anchor=west,xshift=10mm] (g2) at (g1.east)
    {
      \begin{tikzpicture}[x={(0.5,0)},y={(0,0.5)}]
        \node [circle, scale = 0.3, fill](v1) at (0,0){};
        \node [circle, scale = 0.3, fill](v2) at (0,2){};
        \node [circle, scale = 0.3, fill](v3) at (1,1){} [Green];

        \draw (v1) -- (v2);
        \draw (v1) -- (v3) -- (v2) [Green];
      \end{tikzpicture}
    };
    \node[anchor=west,xshift=10mm] (g3) at (g2.east)
    {
      \begin{tikzpicture}[x={(0.5,0)},y={(0,0.5)}]
        \node [circle, scale = 0.3, fill](v1) at (0,0){};
        \node [circle, scale = 0.3, fill](v2) at (0,2){};
        \node [circle, scale = 0.3, fill](v3) at (1,1){};
        \node [circle, scale = 0.3, fill](v4) at (4,0){}[Green];

        \draw (v1) -- (v2);
        \draw (v1) -- (v3) -- (v2);
        \draw (v1) -- (v4) -- (v3) [Green];
      \end{tikzpicture}
    };
    \node[anchor=west,xshift=10mm] (g4) at (g3.east)
    {
      \begin{tikzpicture}[x={(0.5,0)},y={(0,0.5)}]
        \node [circle, scale = 0.3, fill](v1) at (0,0){};
        \node [circle, scale = 0.3, fill](v2) at (0,2){};
        \node [circle, scale = 0.3, fill](v3) at (1,1){};
        \node [circle, scale = 0.3, fill](v4) at (4,0){};
        \node [circle, scale = 0.3, fill](v5) at (4,2){} [Green];

        \draw (v1) -- (v2);
        \draw (v1) -- (v3) -- (v2);
        \draw (v1) -- (v4) -- (v3);
        \draw (v4) -- (v5) -- (v2) [Green];
      \end{tikzpicture}
    };
    \node[anchor=west,xshift=10mm] (g5) at (g4.east)
    {
      \begin{tikzpicture}[x={(0.5,0)},y={(0,0.5)}]
        \node [circle, scale = 0.3, fill](v1) at (0,0){};
        \node [circle, scale = 0.3, fill](v2) at (1,1){};
        \node [circle, scale = 0.3, fill](v3) at (0,2){};
        \node [circle, scale = 0.3, fill](v4) at (4,0){};
        \node [circle, scale = 0.3, fill](v5) at (3,1){}[Green];
        \node [circle, scale = 0.3, fill](v6) at (4,2){};

        \draw (v1) -- (v2) -- (v3) -- (v1);
        \draw (v4) -- (v5) -- (v6)[Green];
        \draw (v4)--(v6);
        \draw (v1) -- (v4);
        \draw (v2) -- (v5)[Green];
        \draw (v3) -- (v6);
      \end{tikzpicture}
    };
    \draw[->,shorten <=1mm,shorten >=1mm] (g1) -- (g2) node[midway,above,font=\scriptsize] {$H_0$};
    \draw[->,shorten <=1mm,shorten >=1mm] (g2) -- (g3) node[midway,above,font=\scriptsize] {$H_0$};
    \draw[->,shorten <=1mm,shorten >=1mm] (g3) -- (g4) node[midway,above,font=\scriptsize] {$H_0$};
    \draw[->,shorten <=1mm,shorten >=1mm] (g4) -- (g5) node[midway,above,font=\scriptsize] {$H_1$};
  \end{tikzpicture}
  \caption{Henneberg moves to construct the prism graph\pointOrNoPoint}
  \label{fig:HennebergMoves}
\end{figure}

\subsection{Tropicalizations of multigraphs}
In this section, we recall the tropicalizations of graphic matroids arising from multigraphs, or tropicalizations of multigraphs in short.  For the sake of efficiency of our implementation, we will consider a coarsening of the usual Bergman fan structure via lattice of flats \cite{ArdilaKlivans2006}.  More information on Bergman fans can be found in \cite[Section 4.2]{MaclaganSturmfels2015}.

\begin{definition}
  \label{def:flatsAndChainsOfFlats}
  Let $\GG$ be a multigraph.  A \emph{flat} of $\GG$ is a subset $F\subseteq E(\GG)$ that is the edge set of a subgraph whose connected components are vertex-induced subgraphs.  A flat $F\subseteq E(\GG)$ is \emph{proper}, if $F\neq E(\GG)$.  From hereon, all our flats will be proper.  
  The \emph{rank} of a flat $F$ is the rank of its signed vertex-edge matrix:
  \begin{equation*}
      \rank(F)\coloneqq 
      \rank\Big( (m_{v_j,e_i})_{v_j\in V(G), e_i\in F} \Big), \text{ where } m_{v_j,e_i} =
      \begin{cases}
          -1 &\text{if } e_i=\{ v_j,v_{j'}\} \text{ for } j<j',\\
          1 &\text{if } e_i=\{v_{j'},v_j\} \text{ for } j'<j,\\
          0 &\text{otherwise.}
      \end{cases}
  \end{equation*}  
  A \emph{chain} of proper flats is a nested sequence of proper flats
  \begin{equation*}
    F_\bullet\colon\quad \emptyset=F_0\subsetneq F_1\subsetneq \dots\subsetneq F_r \quad \big(\!\subsetneq E(\GG)\big),
  \end{equation*}
  and we define its \emph{length} to be $\mathrm{length(F_\bullet)}\coloneqq r$.
\end{definition}

\begin{definition}
  \label{def:tropicalization}
  Let $\GG$ be a multigraph.  Any chain of proper flats $F_\bullet$ on $\GG$ gives rise to a \emph{Bergman cone}
  \begin{equation*}
    \sigma(F_\bullet)\coloneqq \RR_{\geq 0}\cdot \bbone_{F_0}+\RR_{\geq 0}\cdot \bbone_{F_1}+\dots+\RR_{\geq 0}\cdot \bbone_{F_r}+\RR\cdot \bbone_{[m]}\subseteq\RR^m
  \end{equation*}
  where $\bbone_{F_j}\in \{0,1\}^m$ denotes the indicator vector of the flat $F_j$ and $\bbone_{[m]}\coloneqq(1,\dots,1)$ denotes the all-ones vector.
  The \emph{Bergman fan} of the graphic matroid $M_\GG$ is the balanced polyhedral complex in $\RR^m$ defined by
  \begin{align*}
    &\Trop(M_\GG) \coloneqq \{\sigma(F_\bullet)\mid F_\bullet \text{ chain of flats of }\GG\} \quad\text{and}\\[1mm]
    &\mult_{\Trop(M_\GG)}(\sigma(F_\bullet))\coloneqq 1 \text{ for } F_\bullet \text{ maximal}.
  \end{align*}  
  We further define $\Trop(\GG)$ as the following cartesian product over all connected components $\GG'\subseteq\GG$:
  \begin{equation*}
    \Trop(\GG) \coloneqq \prod_{\substack{\GG'\subseteq \GG\\ \text{c.c.}}} \Trop(M_{\GG'})=\Big\{\prod_{\substack{\GG'\subseteq \GG\\ \text{c.c.}}} \sigma_{\GG'} \bigmid \sigma_{\GG'}\in \Trop(M_{\GG'})\Big\},
  \end{equation*}
  where we consider the $\Trop(M_{\GG'})$ and the $\sigma_{\GG'}$ as fans and cones in $\RR^{E(\GG')}$, respectively.  We refer to $\Trop(\GG)$ as the \emph{tropicalization} of $\GG$.
\end{definition}

\begin{figure}[t]
  \centering
  \begin{tikzpicture}
    \node (first)
    {
      \begin{tikzpicture}[x={(0.7,0)},y={(0,0.7)}]
        \node[circle, scale= 0.4, fill](a1) at (-0.6,0){};
        \node[circle, scale= 0.4, fill](a2) at (4.6,0){};
        \node[circle, scale= 0.4, fill](a3) at (2,4){};
        \node(A3) at (2,4) [above]{G};
        
        \draw (a1)--(a2) node[pos = 0.8, below] {$3$};
        \draw (a2)--(a3) node [pos = 0.8, above right, xshift = -3pt] {$2$};
        \draw (a3)--(a1) node[pos = 0.8, above left, xshift = 2pt] {$1$};

        \node[circle, scale = 0.4, fill, color = blue](b1) at (2,1.5){};
        \node(B1) at (2,1.5) [ below right,yshift = 5pt,xshift = -2pt] [blue]
        {\tiny
          $\RR \left(
            \begin{smallmatrix}
              1 \\
              1 \\
              1
            \end{smallmatrix}\right)$
        };
        \node(A1) at (0,0) [above right, xshift = 0pt,yshift = 6pt][blue]{\footnotesize $\Trop(G)$};

        \draw [-stealth] (b1) to node[pos=0.9, left,  color=blue]{$\sigma_{3}$} (2,-2) [blue];
        \draw [-stealth] (b1) to node[pos=0.9, below right ,  color=blue]{$\sigma_{2}$} (5,3.5) [blue];
        \draw [-stealth] (b1) to node[pos=0.9, above right,  color=blue]{$\sigma_{1}$} (-1,3.5) [blue];

        \draw[dashed] (5.25,3.67) -- (9.5,6.5)[blue] node[pos=0.5, below , right, xshift = -6mm, yshift = -10mm](TropFF2){$\Trop(\FF_2)$};
        \node[anchor=north,yshift=-5mm, blue] (SpanFF2) at (TropFF2.south) {$\Span (\bbone_{\{2\}}, \bbone_{\{13\}})$};
        \draw[draw opacity=0] (SpanFF2) -- node[sloped,blue] {$=$} (TropFF2);
        \draw [shorten <=-4.5mm, shorten >=7mm](6.5,4.5)--(9.5,6.5)[blue];
        \draw[dashed] (-1.25,3.67) -- (-5.5,6.5)[blue];
        \draw [shorten <=-4.5mm, shorten >=7mm] (-2.5,4.5)--(-5.5,6.5)[blue] node[pos=0.5, right, xshift = 10pt, yshift = 0pt](TropFF1) {$\Span (\bbone_{\{1\}}, \bbone_{\{23\}})$};
        \node[anchor=south,yshift=5mm,blue] (SpanFF1) at (TropFF1.north) {$\Trop(\FF_1)$};
        \draw[draw opacity=0] (SpanFF1) -- node[sloped,blue] {$=$} (TropFF1);
        \draw[dashed] (2,-2.25) -- (2,-7.5)[blue];
        \draw [shorten <=-4.5mm, shorten >=3.5mm] (2,-4)--(2,-7)[blue] node[pos=0.5, left, xshift = -3mm,yshift = 2mm] (TropFF3) {$\Span (\bbone_{\{3\}}, \bbone_{\{12\}})$};
        \node[anchor=south,yshift=5mm,blue] (SpanFF3) at (TropFF3.north) {$\Trop(\FF_3)$};
        \draw[draw opacity=0] (SpanFF3) -- node[sloped,blue] {$=$} (TropFF3);

        \draw [-stealth] (-0.8,3) to [bend left=40] node[pos=0.5, below,red]{$1$}  (-2,3.8)[red];
        \draw [-stealth] (4.5,3.6) to [bend left=40] node[pos=0.5, above,red]{$2$}  (5.7,4.4)[red];
        \draw [-stealth] (2.25,-1.8) to [bend left=40] node[pos=0.5, right,red]{$3$}  (2.25,-3.2)[red];

        \node[circle, scale= 0.4, fill](c1) at (1,-6.5){};
        \node[circle, scale= 0.4, fill](c2) at (3,-6.5){};
        \node[circle, scale= 0.4, fill](c3) at (3,-6){};
        \node[circle, scale= 0.4, fill](c4) at (3,-4){};
        \node(C1) at (1,-6.5) [left]{$\FF_3$};
        \node[circle, scale= 0.4, fill](c5) at (7.95,6.66){};
        \node[circle, scale= 0.4, fill](c6) at (9.05,5){};
        \node[circle, scale= 0.4, fill](c7) at (7.55,6.39){};
        \node[circle, scale= 0.4, fill](c8) at (6,5.36){};
        \node(C6) at (9.05,5) [below right]{$\FF_2$};
        \node[circle, scale= 0.4, fill](c9) at (-3.94,6.66){};
        \node[circle, scale= 0.4, fill](c10) at (-5.05,5){};
        \node[circle, scale= 0.4, fill](c11) at (-4.65,4.73){};
        \node[circle, scale= 0.4, fill](c12) at (-3.1,3.7){};
        \node(C9) at (-3.94,6.66) [above left]{$\FF_1$};

        \draw (c1)--(c2) node[pos=0.8, below]{$3$};
        \draw (c3)--(c4) node[pos=0.5, right]{$12$};
        \draw (c5)--(c6) node[pos=0.3,above right,xshift = -5pt]{$2$};
        \draw (c7)--(c8) node[pos=0.6,above, yshift = 3pt]{$13$};
        \draw (c9)--(c10) node[pos =1,above, xshift = -2pt, yshift = 2pt] {$1$};
        \draw (c11)--(c12) node[pos = 0.3, below, yshift = -2pt] {$23$};
      \end{tikzpicture}
    };
  \end{tikzpicture}
  \caption{The Laman graph $G$, its full excisions $\FF_1, \FF_2, \FF_3$, and their tropicalizations\pointOrNoPoint}
  \label{fig:running}
\end{figure}
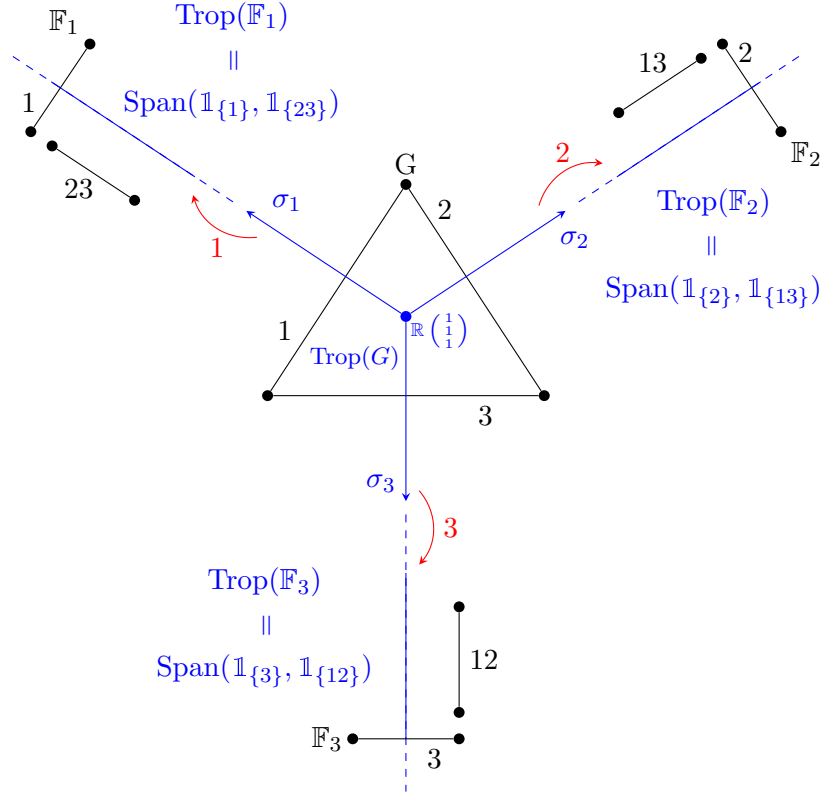

\begin{example}
  \label{ex:tropicalization}
  Let $G$ be the complete graph on $3$ vertices, and for $i=1,2,3$ let $\FF_i$ be the multigraph on $4$ vertices with an isolated edge $i$ and an isolated multiedge $\{1,2,3\}\setminus \{i\}$ as illustrated in \cref{fig:running}.
  \begin{enumerate}
  \item Note that $G$ is connected and has $5$ flats, $\emptyset, \{1\}, \{2\}, \{3\}, \{1,2,3\}=E(G)$, which in turn gives rise to $3$ maximal chains $F_{1,\bullet},F_{2,\bullet},F_{3,\bullet}$ of length $1$:
    \begin{equation*}
      F_i\colon \qquad \emptyset\subsetneq \{i\} \qquad\text{for }i=1,2,3.
    \end{equation*}
    Consequently, $\Trop(G)$ consists of $3$ maximal cones
    \begin{equation*}
      \sigma_i=\RR_{\geq 0}\cdot \bbone_{\{i\}} + \RR\cdot \bbone_{\{1,2,3\}} \qquad\text{for }i=1,2,3.
    \end{equation*}
  \item The $\FF_i$ consists of two connected components, $\FF_i'$ with a single edge $i$ and $\FF_i''$ with a single multiedge $\{1,2,3\}\setminus\{i\}$.  Each connected component only has a single proper flat, which is the empty set. Hence, each connected component only has a single maximal chain of length $0$ consisting only of the empty set.
    Consequently, we have
    \begin{align*}
      \Trop(\FF_i) &= \Trop(M_{\FF_i'})\times\Trop(M_{\FF_i''}) = \Big\{\Span(\underbrace{\bbone_{\{i\}}}_{\in\RR^{\{1\}}})\times\Span(\underbrace{\bbone_{\{1,2,3\}\setminus\{i\}}}_{\in\RR^{\{1,2,3\}\setminus\{i\}}})\Big\}\\[-3mm]
                   &=\Big\{\Span(\underbrace{\bbone_{\{i\}},\bbone_{\{1,2,3\}\setminus\{i\}}}_{\in\RR^3=\RR^{\{1,2,3\}}})\Big\}
    \end{align*}
  \end{enumerate}
\end{example}

In \cref{def:tropicalization}, we defined cones using chains of flats.  This yields the following characterization of their points:

\begin{lemma}
  \label{lem:bergmanConeToChainOfFlats}
  Let $\GG$ be a multigraph and for every connected component $\GG'\subseteq\GG$ let $F_{\GG',\bullet}$ be a proper chain of flats on $\GG'$.
  For $w=(w_1,\dots,w_m)\in\RR^m$, say $\{w_1,\dots,w_m\}=\{\lambda_1,\dots,\lambda_r\}$ for $\lambda_1>\dots>\lambda_r$, we define $F_{\GG',w,i}\coloneqq \{i\in E(\GG')\mid w_i\geq \lambda_j\}$, resulting in the following chain of subsets on $E(\GG')$:
  \begin{equation*}
    F_{\GG',w,\bullet}\colon\quad F_{\GG',w,1}\subseteq F_{\GG',w,2}\subseteq\dots\subseteq F_{\GG',w,r}.
  \end{equation*}
  We then have
  \begin{equation*}
    \prod_{\substack{\GG'\subseteq\GG\\\text{c.c.}}} \sigma(F_{\GG',\bullet}) = \mathrm{cl}\Big(\Big\{w\in\RR^m\bigmid F_{\GG',w,\bullet}=F_{\GG',\bullet} \text{ for all c. c. } \GG'\subseteq\GG\Big\}\Big),
  \end{equation*}
  where $\mathrm{cl}(\cdot)$ denotes euclidean closure.
\end{lemma}
\begin{proof}
  Follows straightforwardly from the definition.
\end{proof}

The correspondence between cones in $\Trop(\GG)$ and chains of flats on the connected components of $\GG$ observed in \cref{ex:tropicalization} can be formalized as follows:

\begin{lemma}
  \label{lem:tropicalization}
  Let $\GG$ be a multigraph.  Then we have:
  \begin{enumerate}
  \item The lineality space of $\Trop(\GG)$ equals
    \begin{equation*}
      L\coloneqq \Span\Big(\bbone_{\ee}\bigmid \ee\in\EE(\GG) \text{ isolated}\Big).
    \end{equation*}
  \item Suppose $\ell\coloneqq \dim(L)$ and $r\coloneqq \dim \Trop(\GG)-\ell$.  For any $k=0,\dots,r$, we define
    \begin{align*}
      \Trop(\GG)(k) & \coloneqq \Big\{\sigma\in\Trop(\GG)\bigmid \dim(\sigma) = \ell+k \Big\} \\[2mm]
      \mathcal C(\GG)(k) & \coloneqq \left\{ (F_{\GG',\bullet})_{\subalign{&\GG'\subseteq \GG\\&\text{c.c.}}} \bigmid
      \begin{array}{c}
        F_{\GG',\bullet} \text{ chain of proper flats on } \EE(\GG') \text{ and} \\
        \sum_{\GG'\subseteq\GG\, c. c.} \mathrm{length}(F_{\GG',\bullet}) = k
      \end{array}
      \right\}.
    \end{align*}
    Then there is a one-to-one correspondence
    \begin{equation*}
      \begin{array}{ccc}
        \Trop(\GG)(k) & \hspace{5mm}\longleftrightarrow & \mathcal C(\GG)(k)\\[2mm]
        \prod_{\subalign{&\GG'\subseteq \GG\\&\text{c.c.}}} \sigma(F_{\GG',\bullet}) & \hspace{5mm}\longleftrightarrow & \phantom{{}_{\subalign{&\GG'\subseteq \GG\\&\text{conn.}\\&\text{c.}}}}\Big(F_{\GG',\bullet}\Big)_{\subalign{&\GG'\subseteq \GG\\&\text{c.c.}}}
      \end{array}
    \end{equation*}
  \end{enumerate}
\end{lemma}
\begin{proof}
  Follows straightforwardly from the definition.
\end{proof}

And from \cref{lem:tropicalization}, we immediately obtain:

\begin{corollary}
  \label{cor:tropicalization}\
  \begin{enumerate}
  \item If $\ee\in\EE(\GG)$ is an isolated multiedge in $\GG$, then $\Trop(\GG)$ is invariant under translation by $\bbone_{\ee}$ in the sense that $\sigma=\sigma+\RR\cdot \bbone_{\ee}$ for all $\sigma\in\Trop(\GG)$.
  \item \label{enumitem:isolatedTriangleAndEdges} If $\GG$ is a disjoint union of multiedges $\ee_1, \dots, \ee_r$, then
    \begin{equation*}
      \Trop(\GG) = \big\{ \Span(\bbone_{\ee_1}, \dots, \bbone_{\ee_r},\bbone_{[m]})\big\}.
    \end{equation*}
  \item \label{enumitem:isolatedEdges} If $\GG$ is a disjoint union of multiedges $\ee_4,\dots,\ee_r$ and a single multitriangle with multiedges $\ee_1,\ee_2,\ee_3$, then
    \begin{equation*}
      \Trop(\GG) = \{\sigma_1,\sigma_2,\sigma_3\} \quad\text{where}\quad \sigma_i\coloneqq \RR_{\geq 0}\cdot \bbone_{\ee_i} + \Span(\bbone_{\ee_4}, \dots, \bbone_{\ee_r},\bbone_{[m]}).
    \end{equation*}
  \end{enumerate}
\end{corollary}

\subsection{Tropical intersection numbers}
In this section, we assume some basic familiarity with balanced polyhedral complexes and recall the concept of their stable intersection and the resulting tropical intersection numbers.  We then use them to define the realization number of Laman graphs.  More information on stable intersections can be found in \cite[Section 3.6]{MaclaganSturmfels2015}.

\begin{definition}
  \label{def:stableIntersectionMS}
  Let $\Sigma_1, \Sigma_2$ be two balanced polyhedral complexes in $\RR^m$. The \emph{stable intersection} of $\Sigma_1$ and $\Sigma_2$ is defined to the polyhedral complex
  \begin{equation*}
    \Sigma_1\wedge\Sigma_2 \coloneqq \Big\{ \sigma_1\cap\sigma_2 \mid \dim(\sigma_1+\sigma_2)=m\Big\}
  \end{equation*}
  together with multiplicities defined by
  \begin{equation*}
    \mult_{\Sigma_1\wedge\Sigma_2}(\sigma_1\cap\sigma_2)\coloneqq \sum_{\tau_1, \tau_2} \mult_{\Sigma_1}(\tau_1)\cdot \mult_{\Sigma_2}(\tau_2) \cdot [N:N_1+N_2]
  \end{equation*}
  where $\tau_i\in\Sigma_i$ such that $(\sigma_1\cap\sigma_2)\subseteq\tau_i$ for $i={1,2}$ and $\tau_1\cap(\tau_2+\varepsilon\cdot u)\neq \emptyset$ for some fixed $u\in\RR^m$ generic and $\varepsilon>0$ sufficiently small, and $N\coloneqq \ZZ^m$ is the standard lattice in $\RR^m$, $N_i\coloneqq L_i\cap\ZZ^m$ are the sublattices induced by $\tau_i$, and $[N:N_1+N_2]$ is the index of sublattice $N_1+N_2$ inside $N$.
  \\Alternatively, by \cite[Proposition 3.6.12]{MaclaganSturmfels2015}, one can show that
  \begin{equation*}
    \Sigma_1\wedge\Sigma_2=\lim_{\varepsilon\rightarrow 0}\Sigma_1\wedge (\Sigma_2+\varepsilon\cdot u) \qquad\text{for } u\in\RR^m\text{ generic}.
  \end{equation*}
\end{definition}




\begin{theorem}[{\cite[Theorem 3.6.10]{MaclaganSturmfels2015}}]
  \label{thm:stableIntersection}
  Let $\Sigma$ and $\Sigma'$ be two balanced polyhedral complexes in $\RR^m$. Then their stable intersection $\Sigma\wedge\Sigma'$ is either empty or a balanced polyhedral complex with
  \begin{equation*}
    \codim(\Sigma\wedge\Sigma') = \codim(\Sigma)+\codim(\Sigma').
  \end{equation*}
  Moreover, if $\codim(\Sigma)+\codim(\Sigma')+l>m$, where $l$ is the dimension of the intersection of the lineality spaces of $\Sigma$ and $\Sigma'$, then $\Sigma\wedge\Sigma'$ is empty.
\end{theorem}

\begin{definition}
  \label{def:tropicalIntersectionProduct}
  Let $\GG$ and $\GG'$ be two multigraphs with $m$ edges such that $\codim(\Trop(\GG))+\codim(\Trop(\GG'))=m-1$, so that the stable intersection $\Trop(\GG)\wedge (-\Trop(\GG'))$ is either empty or one-dimensional by \cref{thm:stableIntersection}, in which case it consists solely of $\Span(1_{[m]})$.  We define their \emph{tropical intersection product} to be
  \forGlasgow
    {
    \begin{equation*}
       \Trop(\GG)\cdot (-\Trop(\GG')) \coloneqq 
    \begin{cases}
            0 & \text{if } \Trop(\GG)\wedge (-\Trop(\GG'))=\emptyset,\\
           \mult_{\Trop(\GG)\wedge (-\Trop(\GG'))}(\Span(1_{[m]})) &\text{otherwise}.
       \end{cases}
    \end{equation*}
    }
  \forArXiv
  {
  \begin{equation*} 
     \Trop(\GG)\cdot (-\Trop(\GG')) \coloneqq 
     \begin{cases}
         0 \hspace{45.5mm} \text{if } \Trop(\GG)\wedge (-\Trop(\GG'))=\emptyset,\\
         \mult_{\Trop(\GG)\wedge (-\Trop(\GG'))}(\Span(1_{[m]})) \hspace{25mm} \text{otherwise}.
     \end{cases}
  \end{equation*}
  }
\end{definition}

We will use tropical intersection numbers to define the realization number of a Laman graph, which are normally defined via structural rigidity properties.  The fact that both numbers coincide is \cite[Theorem 3.8]{ClarkeDewarTrippMaxwellNixonRenSmith2025}.

\begin{definition}
  \label{thm:realisationNumber}
  The \emph{realization number} of a Laman graph $G$ is
  \begin{equation*}
    c_2(M_G) \coloneqq \frac{\Trop(G)\cdot (-\Trop(G))}{2}.
  \end{equation*}
\end{definition}

\subsection{Tropical stars}
Finally, we introduce stars of balanced polyhedral complexes similar to \cite[Definition 2.3.6]{MaclaganSturmfels2015} and show how they can be used for lower bounds of tropical intersection numbers.

\begin{definition}
  \label{def:star}
  Let $\Sigma$ be a balanced polyhedral complex in $\RR^n$, and let $w\in|\Sigma|$ be a point in its support.  Then any $\sigma\in\Sigma$ gives rise to a polyhedral cone describing $\sigma$ around $w$, which is non-empty if and only if $w\in\sigma$:
  \begin{equation*}
    \Star_\sigma(w)\coloneqq \{u\in\RR^n \mid w + \varepsilon\cdot u\in\sigma \text{ for }\varepsilon>0 \text{ sufficiently small}\}.
  \end{equation*}
  The \emph{star} of $\Sigma$ around $w$ is the balanced polyhedral fan in $\RR^n$ given by
  \begin{align*}
    & \Star_\Sigma(w)\coloneqq \{\Star_\sigma(w)\mid \sigma\in\Sigma\} \quad\text{and}\quad \mult_{\Star_\Sigma(w)}(\Star_\sigma(w))\coloneqq \mult_\Sigma(\sigma).
  \end{align*}
  One can show that $\Star_\Sigma(w)$ is indeed balanced, and that $\Star_\Sigma(w)=\Star_\Sigma(w')$ if $w, w' \in\Relint(\tau)$ for some $\tau\in\Sigma$.  Hence, we define
  \begin{equation*}
    \Star_\Sigma(\tau)\coloneqq \Star_\Sigma(w)\quad\text{for any }w\in\Relint(\tau).
  \end{equation*}
\end{definition}

\begin{example}
  \label{ex:star}
  \cref{fig:stars} shows a tropical plane curve $\Sigma$ and its stars.  In particular, it illustrates how consecutive stars of $\Sigma$ are just normal stars of $\Sigma$, e.g., $\Star_{\Star_\Sigma(\tau_0)}(\sigma_1')=\Star_\Sigma(\sigma_1)$.
\end{example}

\begin{figure}[t]
  \centering
  \begin{tikzpicture}
    \node (first)
    {
      \begin{tikzpicture}[x={(0.7,0)},y={(0,0.7)}]
        \node[circle, scale= 0.3, fill, red](a1) at (0,0) []{};
        \node[circle, scale= 0.3, fill, red](a2) at (1.5,1.5)[]{};
        \node[circle, scale= 0.3, fill, red](a3) at (0,-2.1)[]{};
        \node[circle, scale= 0.3, fill, red](a4) at (-2.2,0)[]{};
        \node (a5) at (1.5,3)[]{};
        \node (a6) at (3,1.5)[]{};
        \node (a7) at (3,-2.1)[]{};
        \node (a8) at (-1.5,-3.6)[]{};
        \node (a9) at (-2.2,3)[]{};
        \node (a10) at (-3.6,-1.5)[]{};
        \node (a11) at (0,-3.5) []{$\Sigma$};

        \node(A1) at (0,0) [right, red] {$\tau_0$};
        \node(A2) at (1.5,1.5) [below right, red] {$\tau_3$};
        \node(A3) at (0,-2.1) [below, yshift = -1mm, red] {$\tau_2$};
        \node(A4) at (-2.2,0) [left, red] {$\tau_1$};

        \draw (a1) -- (a2) node[midway, left] {$\sigma_7$};
        \draw (a1) -- (a3)node[midway, left] {$\sigma_4$};
        \draw (a1) -- (a4)node[midway, above] {$\sigma _ 1$};
        \draw (a2) -- (a5) node [midway, left]{$\sigma_6$};
        \draw (a2) -- (a6) node [midway, above]{$\sigma_3$};
        \draw (a3) -- (a8)node [midway, above, left] {$\sigma_9$};
        \draw (a3) -- (a7) node [midway, above]{$\sigma_2$};
        \draw (a4) -- (a9) node [midway, left]{$\sigma_5$};
        \draw (a4) -- (a10) node [midway, left]{$\sigma_8$};

      \end{tikzpicture}
    };
    \node[xshift = -10mm] (second) at (first.west)
    {
      \begin{tikzpicture}[x={(0.7,0)},y={(0,0.7)}]
        \node[circle, scale= 0.3, fill](b1) at (0,0) []{};
        \node(b2) at (1.5,1.5)[]{};
        \node(b3) at (-2.1, 0)[]{};
        \node(b4) at (0,-2.1) [] {};
        \node(b5) at (-0.3,-2.5) []{$\Sigma_0 ' = \Star _ {\Sigma} (\tau _ 0)$};

        \draw (b1)--(b2) node[pos = 0.6, left, xshift = -1mm] {$\sigma_3 '$};
        \draw (b1) -- (b3) node[midway, above, xshift = -1mm] {$\sigma_1 '$};
        \draw (b1) -- (b4) node[midway, right] {$\sigma_2 '$};

        \node[circle, scale= 0.3, fill](c1) at (0,-6) [] {};
        \node(c2) at (2.1,-6)[]{};
        \node(c3) at (0, -3.9)[]{};
        \node(c4) at (-1.5,-7.5) [] {};
        \node(c5) at (-0.3,-8) []{$\Sigma_i ' = \Star _ {\Sigma} (\tau _ i)$};
        \node[anchor=base west] at (c5.base east) {$i=1,2,3$};

        \draw (c1) -- (c2) node[pos = 0.5, above, xshift = -1mm] {$\sigma_1 '$};
        \draw (c1) -- (c3) node[pos = 0.5, left] {$\sigma_2 '$};
        \draw (c1) -- (c4) node[pos = 0.5, below,right,  xshift = 1mm] {$\sigma_3 '$};
   \end{tikzpicture}
    };
    \node[xshift = 30mm] (third) at (first.east)
    {
      \begin{tikzpicture}[x={(0.7,0)},y={(0,0.7)}]
        \node(d1) at (-1.05,-0.5) [] {};
        \node(d2) at (1.05,-0.5)[]{};
        \node(d3) at (2.5,0.25) []{$\Star _ {\Sigma} (\sigma _ i)$};
        \node[anchor=base west] at (d3.base east) []{$i=1,2,3$};
        \node(d5) at (2.5,-1.25) [] {$\Star _ {\Sigma'_j} (\sigma _ 1')$};
        \node[anchor=base west] at (d5.base east) []{$j=1,2$};
        \draw[draw opacity=0] (d5) -- node[sloped] {$=$} (d3);
        \draw (d1) -- (d2) {};

        \node(e1) at (0,-4.9) [] {};
        \node(e2) at (0,-2.9)[]{};
        \node(e3) at (2.5,-3.15) []{$\Star _ {\Sigma} (\sigma _ i)$};
        \node[anchor=base west] at (e3.base east) []{$i=4,5,6$};
        \node(e5) at (2.5,-4.65) [] {$\Star _ {\Sigma'_j} (\sigma _ 2')$};
        \node[anchor=base west] at (e5.base east) []{$j=1,2$};
        \draw[draw opacity=0] (e5) -- node[sloped] {$=$} (e3);
        \draw (e1) -- (e2) {};

        \node(f1) at (-1.05,-8.25) [] {};
        \node(f2) at (1.05,-6.25)[]{};
        \node(f3) at (2.5,-6.5) []{$\Star _ {\Sigma} (\sigma _ i)$};
        \node[anchor=base west] at (f3.base east) []{$i=7,8,9$};

        \node(f5) at (2.5,-8) [] {$\Star _ {\Sigma'_j} (\sigma _ 3')$};
        \node[anchor=base west] at (f5.base east) []{$j=1,2$};
        \draw[draw opacity=0] (f5) -- node[sloped] {$=$} (f3);
        \draw (f1) -- (f2) {};
      \end{tikzpicture}
    };
  \end{tikzpicture}
  \caption{All stars of the polyhedral complex $\Sigma$\pointOrNoPoint}
  \label{fig:stars}
\end{figure}
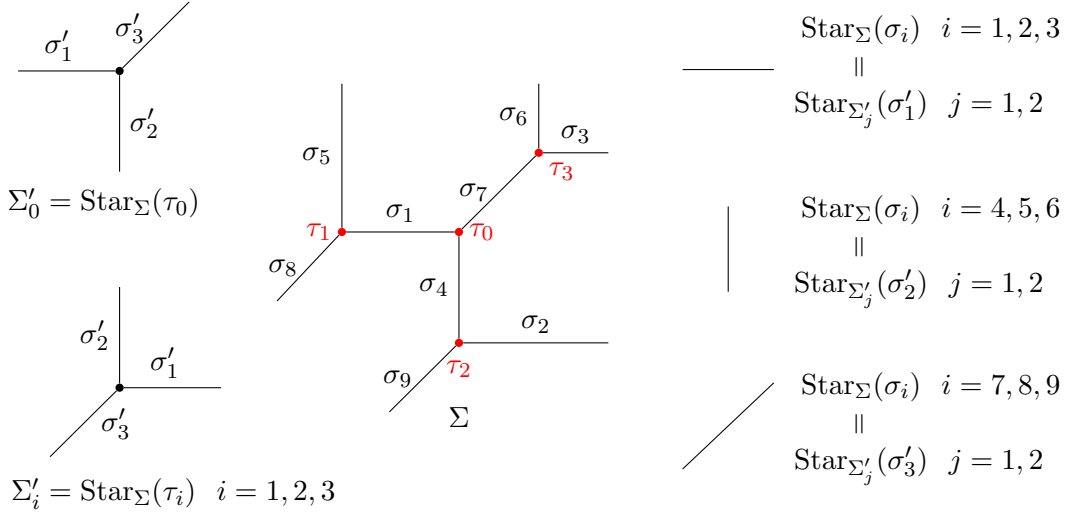

As illustrated in \cref{ex:star}, the star is combinatorially easier than the original polyhedral complex, hence it is not surprising that replacing polyhedral complexes by their stars cannot increase the tropical intersection product.

\begin{lemma}
  \label{lem:intersectionProductAndStar}
  Let $\Sigma$, $\Sigma'$ be two balanced polyhedral complexes of complementary dimension.  Then
  \begin{equation*}
    \Sigma\cdot\Sigma'\geq \Star_{\Sigma}(\sigma)\cdot \Sigma' \qquad\text{ for all }\sigma\in\Sigma.
  \end{equation*}
\end{lemma}
\begin{proof}
    Without loss of generality, we may assume that $w\coloneqq 0\in\Relint(\sigma)$.

  For $t>0$, let $t\cdot\Sigma$ be the balanced polyhedral complex with polyhedra $t\cdot\sigma$, where $\sigma\in\Sigma$ and $t\cdot (\ldots)$ denotes linear scaling by $t$, and multiplicities $\mult_{t\cdot\Sigma}(t\cdot\sigma)=\mult_{\Sigma}(\sigma)$.  Then $(t\cdot \Sigma)\cdot\Sigma'=\Sigma\cdot\Sigma'$ for all $t>0$, and $t\cdot \Sigma$ converges to $\Star_{\Sigma}(w)$ as $t$ goes to $\infty$. Note that the stable intersection points of $t\cdot\Sigma\wedge\Sigma'$ varies continuously in $t$.

  Let further $r_t\coloneqq\min(\|u\|\mid u\in \sigma, \sigma\in t\cdot\Sigma, 0\notin\sigma)$ denote the minimal distance between $0$ and all polyhedra of $t\cdot \Sigma$ not containing $0$, and let $B_t$ denote the ball around $0$ of radius $r_t$.  Then $t\cdot \Sigma$ and $(t+s)\cdot\Sigma$ coincide inside $B_t$ for all $s>0$, and thus $w\in |t\cdot \Sigma\wedge\Sigma'|\cap B_t$ implies $w\in\Star_{\Sigma}(w)\wedge\Sigma'$.  Moreover, $B_t$ converges to $\RR^n$ as $t$ goes to $\infty$.

  Then, as $t$ goes to $\infty$, any intersection point of $t\cdot\Sigma\wedge\Sigma'$ either falls into $B_t$, becoming an intersection point of $\Star_{\Sigma}(w)\cdot\Sigma'$ or diverges to infinity. This shows the claim.
\end{proof}


%% file: excision.tex
\section{Graph excisions and tropical stars}\label{sec:excision}
In this section, we introduce graph excisions, which are important as they correspond to stars of the tropicalization.  These will play a central role hereon.

\begin{definition}
  \label{def:excision}
  Let $\GG$ be a multigraph.  Let $\ee\in \EE(G)$ be one of its non-isolated multiedges, and let $v_1,v_2\in V(\GG)$ be the two vertices it connects.  We define a new graph $\GG\curvearrowright \ee$, or \emph{$\GG$ excised $\ee$} in words, to be the multigraph with vertex set $V(\GG\curvearrowright \ee) \coloneqq V(\GG)\sqcup\{v_{12}\}$ and edge set $E(\GG\curvearrowright \ee)$ consisting of three types of edges
  \begin{enumerate}
  \item $e\in E(\GG\curvearrowright \ee)$ if $e\in E(G)$ with $e\in \ee$,
  \item $e\in E(\GG\curvearrowright \ee)$ if $e\in E(G)$ with $v_1\notin e$ and $v_2\notin e$,
  \item $\{v,v_{12}\}\in E(\GG\curvearrowright \ee)$ if $\{v,v_1\}\in E(\GG)$ or $\{v,v_2\}\in E(\GG)$.
  \end{enumerate}
\end{definition}

\begin{example}
  Let $G$ be the complete graph on $3$ vertices from \cref{ex:tropicalization}.  Then $G$ has $3$ possible excisions, all of which consists of isolated multiedges, see \cref{fig:running}:
  \begin{equation*}
    \EE( G\curvearrowright 1 ) = \big\{ \{1\}, \{23\} \big\},\quad
    \EE( G\curvearrowright 2 ) = \big\{ \{2\}, \{13\} \big\},\quad\text{and}\quad
    \EE( G\curvearrowright 3 ) = \big\{ \{3\}, \{12\} \big\}.
  \end{equation*}
\end{example}

\begin{example}
  \label{ex:excision}
  Let $G$ be the complete graph on $4$ vertices without an edge.  \cref{fig:excisions} shows $G$ and all its excisions. The labels of the curved arrows represent the multiedges being excised, so that $G\curvearrowright 3\curvearrowright 12 = G\curvearrowright 3\curvearrowright 45$ and $G\curvearrowright 1\curvearrowright 5 = G\curvearrowright 5\curvearrowright 1$.
\end{example}

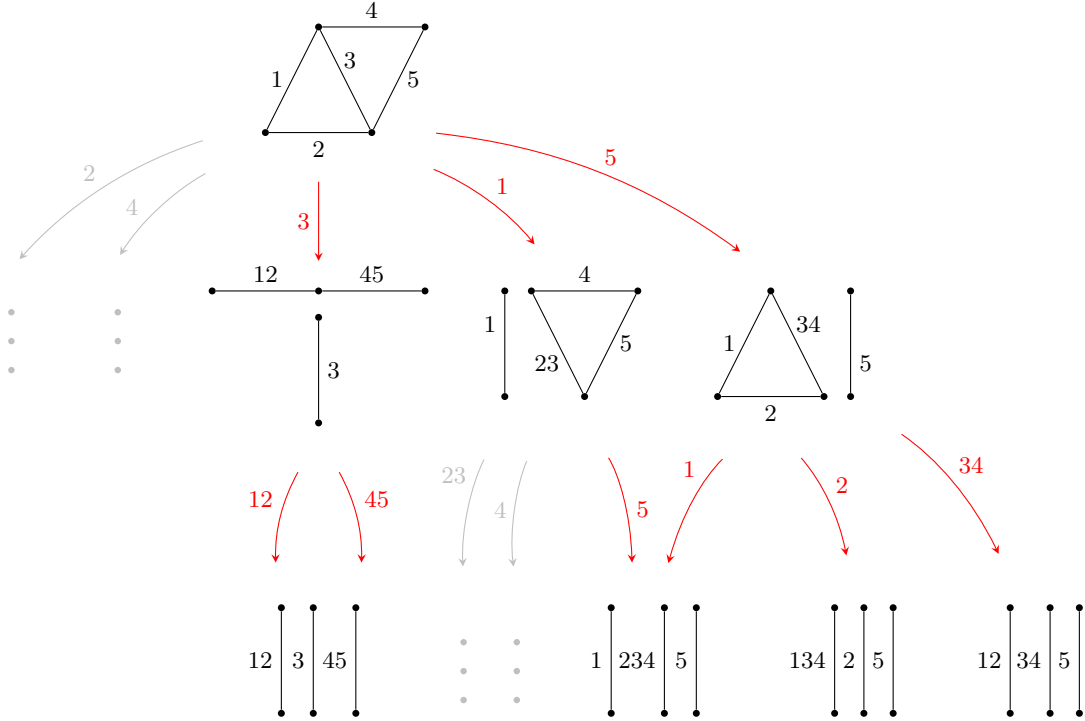
\begin{figure}[t]
  \begin{tikzpicture}[every node/.style = {font=\footnotesize}]
    \node (first)
    {
    };
    \node[anchor=north east] (second) at (first.south)
    {
    };
    \node[anchor=north] (third) at (second.south)
    {
      \begin{tikzpicture}[x={(0.7,0)},y={(0,0.7)}]

        \node[circle, scale= 0.3, fill](a1) at (5.5,14)[]{};
        \node[circle, scale= 0.3, fill](a2) at (7.5,14)[]{};
        \node[circle, scale= 0.3, fill](a3) at (6.5,16)[]{};
        \node[circle, scale= 0.3, fill](a4) at (8.5,16)[]{};

        \draw (a1)--(a2) node[midway, below] {$2$};
        \draw (a2)--(a3) node [pos = 0.7,right] {$3$};
        \draw (a3)--(a1) node[midway, left] {$1$};
        \draw (a4)--(a2) node[midway, right] {$5$};
        \draw (a4)--(a3) node[midway, above] {$4$};

        \node[](b1) at (0.5,9){};
        \node[](b2) at (0.5,11.5){};
        \node[](b3) at (2.5,9){};
        \node[](b4) at (2.5,11.5){};
        \node[circle, scale= 0.3, fill](b5) at (4.5,11){};
        \node[circle, scale= 0.3, fill](b6) at (6.5,11){};
        \node[circle, scale= 0.3, fill](b7) at (8.5,11){};
        \node[circle, scale= 0.3, fill](b8) at (6.5,10.5){};
        \node[circle, scale= 0.3, fill](b9) at (6.5,8.5){};
        \node[circle, scale= 0.3, fill](b10) at (10,11){};
        \node[circle, scale= 0.3, fill](b11) at (10,9){};
        \node[circle, scale= 0.3, fill](b12) at (10.5,11){};
        \node[circle, scale= 0.3, fill](b13) at (12.5,11){};
        \node[circle, scale= 0.3, fill](b14) at (11.5,9){};
        \node[circle, scale= 0.3, fill](b15) at (14,9){};
        \node[circle, scale= 0.3, fill](b16) at (15,11){};
        \node[circle, scale= 0.3, fill](b17) at (16,9){};
        \node[circle, scale= 0.3, fill](b18) at (16.5,9){};
        \node[circle, scale= 0.3, fill](b19) at (16.5,11){};

        \path (b1) -- (b2) node [ gray!50, font=\Huge, midway, sloped] {$\dots$};
        \path (b3) -- (b4) node [gray!50, font=\Huge, midway, sloped] {$\dots$};
        \draw (b5)--(b6) node[midway, above]{$12$};
        \draw (b6)--(b7)node[midway, above]{$45$};
        \draw (b8)--(b9)node[midway, right]{$3$};
        \draw (b10)--(b11)node[pos= 0.3, left]{$1$};
        \draw (b12)--(b13) node[midway, above]{$4$};
        \draw (b13)--(b14)node[midway, right]{$5$};
        \draw (b14)--(b12)node[pos = 0.3, left]{$23$};
        \draw (b15)--(b16)node[pos = 0.5, left]{$1$};
        \draw (b16)--(b17)node[pos = 0.3, right]{$34$};
        \draw (b17)--(b15)node[pos = 0.5, below]{$2$};
        \draw (b18)--(b19)node[pos = 0.3, right]{$5$};

        \node[](c1) at (9,3){};
        \node[](c2) at (9,5){};
        \node[](c3) at (10,3){};
        \node[](c4) at (10,5){};
        \node[circle, scale= 0.3, fill](c5) at (5.8,3){};
        \node[circle, scale= 0.3, fill](c6) at (5.8,5){};
        \node[circle, scale= 0.3, fill](c7) at (6.4,3){};
        \node[circle, scale= 0.3, fill](c8) at (6.4,5){};
        \node[circle, scale= 0.3, fill](c9) at (7.2,3){};
        \node[circle, scale= 0.3, fill](c10) at (7.2,5){};
        \node[circle, scale= 0.3, fill](c11) at (12,3){};
        \node[circle, scale= 0.3, fill](c12) at (12,5){};
        \node[circle, scale= 0.3, fill](c13) at (13,5){};
        \node[circle, scale= 0.3, fill](c14) at (13,3){};
        \node[circle, scale= 0.3, fill](c15) at (13.6,3){};
        \node[circle, scale= 0.3, fill](c16) at (13.6,5){};
        \node[circle, scale= 0.3, fill](c17) at (16.2,3){};
        \node[circle, scale= 0.3, fill](c18) at (16.2,5){};
        \node[circle, scale= 0.3, fill](c19) at (16.75,3){};
        \node[circle, scale= 0.3, fill](c20) at (16.75,5){};
        \node[circle, scale= 0.3, fill](c21) at (17.3,3){};
        \node[circle, scale= 0.3, fill](c22) at (17.3,5){};
        \node[circle, scale= 0.3, fill](c23) at (19.5,3){};
        \node[circle, scale= 0.3, fill](c24) at (19.5,5){};
        \node[circle, scale= 0.3, fill](c25) at (20.25,3){};
        \node[circle, scale= 0.3, fill](c26) at (20.25,5){};
        \node[circle, scale= 0.3, fill](c27) at (20.8,3){};
        \node[circle, scale= 0.3, fill](c28) at (20.8,5){};

        \path (c1) -- (c2) node [gray!50, font=\Huge, midway, sloped] {$\dots$};
        \path (c3) -- (c4) node [gray!50, font=\Huge, midway, sloped] {$\dots$};
        \draw (c5)--(c6)node[pos = 0.5, left]{$12$};
        \draw (c7)--(c8)node[pos = 0.5, left]{$3$};
        \draw (c9)--(c10)node[pos = 0.5, left]{$45$};
        \draw (c11)--(c12)node[pos = 0.5, left]{$1$};
        \draw (c13)--(c14)node[pos = 0.5, left]{$234$};
        \draw (c15)--(c16)node[pos = 0.5, left]{$5$};
        \draw (c17)--(c18)node[pos = 0.5, left]{134};
        \draw (c19)--(c20)node[pos = 0.5, left]{$2$};
        \draw (c21)--(c22)node[pos = 0.5, left]{$5$};
        \draw (c23)--(c24)node[pos = 0.5, left]{$12$};
        \draw (c25)--(c26)node[pos = 0.5, left]{$34$};
        \draw (c27)--(c28)node[pos = 0.5, left]{$5$};

        \draw [-stealth,shorten <=5mm, shorten >=5mm] (5,14) to [bend right = 15] node[above left,  color=gray!50]{$2$} (0.4,11) [gray!50];
        \draw [-stealth,shorten <=5mm, shorten >=5mm] (5,13.5) to [bend right = 15] node[above left, pos = 0.6,  color=gray!50]{$4$} (2.4,11) [gray!50];
        \draw [-stealth] (6.5,13) to node[pos=0.5, left,  color=red]{$3$} (6.5,11.5) [red];
        \draw [-stealth,shorten <=5mm, shorten >=5mm] (8,13.5) to [bend left = 15] node[above right,  color=red]{$1$} (11,11.25) [red];
        \draw [-stealth,shorten <=5mm, shorten >=5mm] (8,14) to [bend left = 15] node[above right,  color=red]{$5$} (15,11.25) [red];

        \draw [-stealth,shorten <=2mm, shorten >=2mm] (6.25,7.75) to [bend right = 15] node[above left,  color=red]{$12$} (5.7,5.5) [red];
        \draw [-stealth, shorten <=2mm, shorten >=2mm] (6.75,7.75) to [bend left = 15] node[above right,  color=red]{$45$} (7.3,5.5) [red];
        \draw [-stealth,shorten <=6mm, shorten >=5mm] (10,8.5) to [bend right = 15] node[above left,pos=0.4,  color=gray!50]{$23$} (9.25,5) [gray!50];
        \draw [-stealth,shorten <=6mm, shorten >=5mm] (10.75,8.5) to [bend right = 15] node[above left,pos=0.6,  color=gray!50]{$4$} (10.25,5) [gray!50];
        \draw [-stealth,shorten <=6mm, shorten >=2mm] (11.5,8.5) to [bend left = 15] node[above right, pos=0.7, color=red]{$5$} (12.4,5.5) [red];
        \draw [-stealth,shorten <=5mm, shorten >=2mm] (14.6,8.25) to [bend right = 15] node[above left, pos=0.4, color=red]{$1$} (13,5.5) [red];
        \draw [-stealth,shorten <=4mm, shorten >=3mm] (15.2,8.2) to [bend left = 15] node[above right, color=red]{$2$} (16.5,5.5) [red];
        \draw [-stealth,shorten <=6mm, shorten >=3.5mm] (16.75,8.7) to [bend left = 15] node[above right, color=red]{$34$} (19.5,5.5) [red];
      \end{tikzpicture}
    };
  \end{tikzpicture}
  \caption{A Laman graph $G$ and its (iterative) excisions\pointOrNoPoint}
  \label{fig:excisions}
\end{figure}

We can formalize the observation of commuting excisions in \cref{ex:excision} as follows:
\begin{lemma}
  \label{lem:excisionsCommuting}
  Let $\GG$ be a multigraph, and let $\ee_1,\ee_2\in\EE(\GG)$ be two of its multiedges on two different connected components.  Then
  \begin{equation*}
    \GG\curvearrowright\ee_1\curvearrowright\ee_2 = \GG\curvearrowright\ee_2\curvearrowright\ee_1.
  \end{equation*}
\end{lemma}
\begin{proof}
  Let $v_1,v_2\in V(\GG)$ be the vertices $\ee_1$ connects, and let $w_1,w_2\in V(\GG)$ be the vertices $\ee_2$ connects. Then $\GG \curvearrowright \ee_1 \curvearrowright \ee_2$ has vertex set $V(\GG) \sqcup \{v_{12}\} \sqcup \{w_{12}\}$ and edge set $E(\GG \curvearrowright \ee_1 \curvearrowright \ee_2)$ with $\{v, v_{12}\}$ and $\{w,w_{12}\} \in E(\GG \curvearrowright \ee_1 \curvearrowright \ee_2)$ as in \cref{def:excision}. $\GG \curvearrowright \ee_2 \curvearrowright \ee_1$ has vertex set $V(\GG) \sqcup \{w_{12}\} \sqcup \{v_{12}\}$ and edge set $E(\GG \curvearrowright \ee_2 \curvearrowright \ee_1)$ with $\{v, v_{12}\}$ and $\{w,w_{12}\}$ also in $E(\GG \curvearrowright \ee_1 \curvearrowright \ee_2)$. Since the vertex sets and edge sets are the same, the two graphs are the same.
\end{proof}

Note that there is a straightforward relationship between chains of excisions and chains of flats:

\begin{lemma}
  \label{lem:excisionAndFlat}
  Let $\GG$ be a multigraph and let $\ee\in \EE(\GG)$ be one of its multiedges, i.e., a rank-$1$ flat of $\GG$.  Then any multiedge $\ee'\in\EE(\GG\curvearrowright \ee)$, $\ee'\neq\ee$ gives rise to a rank-$2$ flat $\ee\cup\ee'\subsetneq E(\GG)$, and any rank-$2$ flat $F\subsetneq E(\GG)$, $\ee\subsetneq F$, gives rise to a multiedge $F\setminus\ee\in\EE(\GG\curvearrowright\ee)$.
\end{lemma}
\begin{proof}
  Follows directly from \cref{def:excision}:  Let $v_1,v_2\in V(\GG)$ denote the vertices of $\ee$ and let $v_{12}\in V(\GG\curvearrowright\ee)$ denote the vertex added by the excision.

  Consider a multiedge $\ee'\in \EE(\GG\curvearrowright\ee)$, $\ee'\neq\ee$.  If $v_{12}\notin V(\ee')$, then $\ee'$ is a multiedge of $\GG$ that is disjoint of $\ee$. If $v_{12}\in V(\ee')$, say $\ee'$ is the multiedge connecting $v_{12}$ and $v_3$, then $\ee'$ consists of all edges of $\GG$ connecting $v_3$ to either $v_1$ or $v_2$.  In both cases, $\ee\cup\ee'$ is a rank-$2$ flat containing $\ee$.

  Conversely, let $F$ be a rank-$2$ flat containing $\ee$.  Then either $F$ is a disjoint union of two multiedges $\ee$ and $F\setminus \ee$ or there is a vertex $v_3\in V(\GG)$ such that $F$ contains $\ee$ and all edges $F\setminus \ee\subseteq E(\GG)$ connecting $v_3$ to either $v_1$ or $v_2$.  In both cases, $F\setminus \ee$ is a multiedge of $\EE(\GG\curvearrowright\ee)$.
\end{proof}

Applying \cref{lem:excisionAndFlat} iteratively yields:

\begin{corollary}
  \label{cor:excisionsAndChains}
  Let $\GG$ be a multigraph.  Then any chain of excisions $\GG\curvearrowright\ee_1\curvearrowright\dots\curvearrowright\ee_r$ gives rise to a chain of flats $\emptyset = F_0\subsetneq \ee_1 \subsetneq \ee_1\cup \ee_2 \subsetneq\dots\subsetneq \ee_1\cup\dots\cup\ee_r$ with $\rank(\ee_1\cup\dots\cup\ee_j)=j$, and any chain of flats $\emptyset=F_0\subsetneq F_1\subsetneq\dots\subsetneq F_r$ with $\rank(F_j)=j$ gives rise to a chain of excisions $\GG\curvearrowright F_1\curvearrowright F_2\setminus F_1\curvearrowright \dots\curvearrowright F_r\setminus F_{r-1}$.
\end{corollary}

\begin{corollary}
  \label{cor:excisionsAndChains2}
  Let $\GG$ be a multigraph, and fix a chain of excisions $\HH\coloneqq \GG\curvearrowright\ee_1\curvearrowright\dots\curvearrowright\ee_r$ and its corresponding chain of flats $F_\bullet\colon \emptyset=F_0\subsetneq F_1\subsetneq\dots\subsetneq F_r$ from \cref{cor:excisionsAndChains}.  Then any chain of flats $F_\bullet'$ on $\HH$ beginning with $F_{i}'=\ee_i$ is a chain of flats on $\GG$ extending $F_\bullet$, and vice versa.
\end{corollary}

\begin{lemma}
  \label{lem:excisionToStar}
  Let $\GG$ be a multigraph and let $\HH\coloneqq \GG\curvearrowright\ee_1\curvearrowright\dots\curvearrowright\ee_r$ arise from a chain of excisions.  Let $\GG_1,\dots,\GG_s$ be the connected components of $\GG$.  By \cref{lem:excisionsCommuting}, we may reorder $\ee_1,\dots,\ee_r$ so that $\EE(\GG_k)\supseteq \{\ee_{j_{k-1}+1},\dots,\ee_{j_k}\}$ for $0=j_0<\dots<j_s=r$. Then the following is a valid cone in $\Trop(\GG)$:
  \begin{align*}
    \sigma \coloneqq &\RR_{\geq 0}\cdot \bbone_{\ee_1}+\RR_{\geq 0}\cdot \bbone_{\ee_1\cup\ee_2}+\dots+\RR_{\geq 0}\cdot \bbone_{\ee_1\cup\dots\cup\ee_{j_1}}\\
    &\quad +\RR_{\geq 0}\cdot \bbone_{\ee_{j_1+1}}+\RR_{\geq 0}\cdot \bbone_{\ee_{j_1+1}\cup\ee_{j_1+2}}+\dots+\RR_{\geq 0}\cdot \bbone_{\ee_{j_1+1}\cup\dots\cup\ee_{j_2}}\\
    &\hspace{4.75mm} \vdots\\
    &\quad +\RR_{\geq 0}\cdot \bbone_{\ee_{j_{s-1}+1}}+\RR_{\geq 0}\cdot \bbone_{\ee_{j_{s-1}+1}\cup\ee_{j_{s-1}+2}}+\dots+\RR_{\geq 0}\cdot \bbone_{\ee_{j_{s-1}+1}\cup\dots\cup\ee_{j_s}}\\
    &\quad +\RR\cdot\bbone_{[m]}\in \Trop(\GG).
  \end{align*}
  Moreover, $\Trop(\HH)=\Star_{\Trop(\GG)}(\sigma)$.
\end{lemma}
\begin{proof}
  We may assume without loss of generality that $\GG$ is connected, which means that the only connected components of $\HH$ are $\HH_1,\dots,\HH_r$ with $E(\HH_j)=\ee_j$, and a component $\HH'\subseteq\HH$ with $E(\HH')=E(\GG)\setminus E$ for $E\coloneqq \ee_1\cup\dots\cup\ee_r$.  By \cref{cor:excisionsAndChains}, the following is a valid chain of flats on $\GG$:
  \begin{equation*}
    F_{E,\bullet}\colon\;\; \emptyset\subsetneq \ee_1\subsetneq \ee_1\cup\ee_2\subsetneq\dots\subsetneq \ee_1\cup\dots\cup\ee_r \subseteq E.
  \end{equation*}
  Hence $\sigma=\sigma(F_{E,\bullet})\in\Trop(\GG)$.  It remains to show that $\Trop(\HH)=\Star_{\Trop(\GG)}(\sigma)$.
  
  For the ``$\subseteq$'' inclusion, let $\tau\in\Trop(\HH)$ be a maximal cone.  By definition, $\tau$ arises from a maximal chain of flats on each connected component of $\HH$, which we can rearrange to become a single maximal chain of flats on $\HH$ beginning with the $\ee_i$:
  \begin{equation*}
    F_{[m],\bullet}\colon\;\; \emptyset\subsetneq \ee_1\subsetneq \dots\subsetneq \ee_1\cup\dots\cup\ee_r\subsetneq E\cup F_1\subsetneq E\cup F_2\subsetneq\dots\subsetneq E\cup F_\ell\subsetneq [m].
  \end{equation*}

  By \cref{cor:excisionsAndChains2}, this is also a valid maximal chain of flats on $\GG$ which corresponds to a maximal cone $\sigma(F_{[m],\bullet})\in\Trop(\GG)$.  We now claim that $\tau=\Star_{\sigma(F_{[m],\bullet})}(v)$ for any $v\in\mathrm{Relint}(\sigma)$, say $v = \bbone_{\ee_1}+\dots+\bbone_{\ee_1\cup\dots\cup\ee_r}$.

  Let $u=(u_i)_{i\in[m]}\in \mathrm{Relint}(\tau)$, i.e., $(u_i)_{i\in\ee_j}$ induces the chains $\emptyset\subsetneq \ee_j$ on the $\HH_j$ and $(u_{i})_{i\notin E}$ induces the chain $\emptyset\subsetneq F_1\subsetneq \dots\subsetneq F_\ell$ on $\HH'$.  Thus we have
  \begin{enumerate}
  \item $u_{i_1}=u_{i_2}$ if $i_1,i_2\in\ee_j$ for some $j=1,\dots,r$,
  \item $u_{i_1}\geq u_{i_2}$ if $i_1\in F_{j_1}\setminus F_{j_1-1}$ and $i_2\in F_{j_2}\setminus F_{j_2-1}$ with $j_1\leq j_2$.
  \end{enumerate}
  We now show that for $\varepsilon>0$ sufficiently small, $w\coloneqq v+\varepsilon\cdot u$ induces $F_{[m],\bullet}$ on $\GG$.
  For that let $w_{i_1}$ and $w_{i_2}$ be two coordinates of $w$ for $i_1, i_2\in [m]$, $i_1\neq i_2$.  We now consider three cases:

  \noindent
  If $i_1,i_2\in E$, say $i_1\in \ee_{j_1}$ and $i_2\in \ee_{j_2}$ with $j_1\leq j_2$, then we have
  \begin{equation*}
    w_{i_1}=v_{i_1}+ \varepsilon\cdot u_{i_1} = r+1-j_1+ \varepsilon\cdot u_{i_1} \geq r+1-j_2 + \varepsilon\cdot u_{i_2}= v_{i_1}+\varepsilon\cdot u_{i_2}=w_{i_2}.
  \end{equation*}
  If $i_1\in E$, say $i_1\in\ee_j$, and $i_2\notin E$, then we have
  \begin{equation*}
    w_{i_1}=r+1-j + \varepsilon\cdot u_{i_1} > \varepsilon\cdot u_{i_2} = w_{i_2}.
  \end{equation*}
  If $i_1,i_2\notin E$, say $i_1\in F_{j_1} \setminus F_{j_1-1}$ and $i_2\in F_{j_2} \setminus F_{j_2-1}$ for some $j_1\leq j_2$, then we have
  \begin{equation*}
    w_{i_1}=\varepsilon\cdot u_{i_1} \geq \varepsilon\cdot u_{i_2}=w_{i_2}.
  \end{equation*}

  Conversely, let $u\in\Star_{\sigma(F_{[m],\bullet})}(v)$, i.e., $w=v+\varepsilon\cdot u$ induces $F_{[m],\bullet}$ on $\GG$.  Then we have
  \begin{enumerate}
  \item \label{enumitem:relevantCondition1}$w_{i_1}=w_{i_2}$ for all $i_1,i_2\in \ee_j$ and all $j=1,\dots,r$,
  \item $w_{i_1}\geq w_{i_2}$ if either
    \begin{enumerate}
    \item $i_1\in \ee_{j_1}$ and $i_2\in\ee_{j_2}$ with $j_1\leq j_2$,
    \item $i_1\in E$ and $i_2\notin E$,
    \item \label{enumitem:relevantCondition2c} $i_1\in F_{j_1}\setminus F_{j_1-1}$ and $i_2\in F_{j_2}\setminus F_{j_2-1}$ with $j_1\leq j_2$.
    \end{enumerate}
  \end{enumerate}
  By Condition~\eqref{enumitem:relevantCondition1}, the $(u_i)_{i\in\ee_j}$ induce the chains $\emptyset\subsetneq \ee_j$ on $\HH_j$ and, by Condition \eqref{enumitem:relevantCondition2c}, $(u_i)_{i\notin E}$ induces the chain $\emptyset \subsetneq F_1 \subsetneq F_2\subsetneq \dots\subsetneq F_\ell$ on $\HH'$.  Hence $u\in\tau$.

  The ``$\supseteq$'' inclusion follows similarly.  Consider a relative interior point $v\in\mathrm{Relint}(\sigma)$, say $v = \bbone_{\ee_1}+\dots+\bbone_{\ee_1\cup\dots\cup\ee_r}$, and let $\Star_{\sigma(F_{[m],\bullet})}(v)\in \Star_{\Trop(\GG)}(\sigma)$ be a maximal cone that arises from a maximal chain of flats $F_{[m],\bullet}$ extending $F_{E,\bullet}$,
  \begin{equation*}
    F_{[m],\bullet}\colon\;\; \emptyset\subsetneq \ee_1\subsetneq \dots\subsetneq \ee_1\cup\dots\cup\ee_r\subsetneq E\cup F_1\subsetneq E\cup F_2\subsetneq\dots\subsetneq E\cup F_\ell\subsetneq [m].
  \end{equation*}

  By \cref{cor:excisionsAndChains2}, this is also a valid maximal chain of flats on $\HH$ which we can restrict to maximal chains on the connected components of $\HH$: $F_{\HH_j,\bullet}$ on the $\HH_j$ for $j=1,\dots,r$ and $F_{\HH',\bullet}$ on $\HH'$.  This gives rise to a maximal cone $\tau\coloneqq \sigma((F_{\HH_1,\bullet},\dots,F_{\HH_r,\bullet},F_{\HH',\bullet}))\in\Trop(\HH)$.  We now claim that $\tau=\Star_{\sigma(F_{[m],\bullet})}(v)$.  This can be shown using the same steps as before.
\end{proof}

\begin{lemma}
  \label{lem:starToExcision}
  Let $\GG$ be a multigraph and let $\sigma\in\Trop(\GG)$ be a cone in the tropicalization arising from the chains of flats $(F_{\GG',\bullet})_{\GG'\subseteq\GG \text{ conn. comp.}}$.  Pick any order on the connected components $\{\GG'\subseteq\GG \mid \GG' \text{ conn. comp.}\} = \{\GG'_1,\dots,\GG'_s\}$ and let $r_j\coloneqq\mathrm{length}(F_{\GG'_j,\bullet})$.  Then the following is a valid chain of multiedge excisions of $\GG$:
  \begin{align*}
    \HH\coloneqq \GG &\curvearrowright F_{\GG'_1,1} \curvearrowright F_{\GG'_1,2}\setminus F_{\GG'_1,1} \curvearrowright F_{\GG'_1,r_1}\setminus F_{\GG'_1,r_1-1}\\
                     &\curvearrowright F_{\GG'_2,1} \curvearrowright F_{\GG'_2,2}\setminus F_{\GG'_2,1} \curvearrowright F_{\GG'_2,r_2}\setminus F_{\GG'_2,r_2-1}\\
    &\curvearrowright \dots\curvearrowright \\
                     &\curvearrowright F_{\GG'_k,1} \curvearrowright F_{\GG'_k,2}\setminus F_{\GG'_k,1} \curvearrowright F_{\GG'_k,r_k}\setminus F_{\GG'_1,r_k-1}.
  \end{align*}
  Moreover, $\Trop(\HH)=\Star_{\Trop(\GG)}(\sigma)$.
\end{lemma}
\begin{proof}
  The proof is similar to that of \cref{lem:excisionToStar}.  The fact that that $\HH$ arises from a valid chain of excisions of $\GG$ follows from \cref{cor:excisionsAndChains}.  And the bijection between the cones of $\Trop(\HH)$ and the cones of $\Star_{\Trop(\GG)}(\sigma)$ stems from the bijection between the chains of flats in \cref{cor:excisionsAndChains2}.
\end{proof}

We conclude this section by introducing two important types of excised graphs.

\begin{definition}
  \label{def:excisionTriangle}
  Let $G$ be a Laman graph and let $\HH$ arise from a chain of excisions,
  \begin{equation*}
    \HH\coloneqq G\curvearrowright\ee_1\curvearrowright\ee_2\dots\curvearrowright\ee_r.
  \end{equation*}
  We say that $\HH$ is \emph{excised fully}, if it is a disjoint union of multiedges.  And we say that $\HH$ is an \emph{excised triangle}, if it is a disjoint union of multiedges and a single multitriangle.
\end{definition}

\begin{lemma}
  \label{lem:excisedTriangles}
  All Laman graphs $G$ with at least $3$ vertices have excised triangles $\HH$.
\end{lemma}

\begin{proof}
  Using \cref{thm:henneberg}, we may perform an induction on the Henneberg moves.
  Clearly, the Laman graph on $3$ edges is already an excised triangle. Let $G'$ be a Laman graph with an excised triangle $\HH'$:
  \begin{equation*}
    \HH'=G'\curvearrowright\ee_1\curvearrowright\dots\curvearrowright\ee_r
  \end{equation*}
  Now let $G$ be obtained from $G'$ via a Henneberg move, and denote the added edges by $e_{m+1}, e_{m+2}$.  Note that $\ee_1,\dots,\ee_r$ remains a valid chain of excisions, and that $\HH'$ is a subgraph of $\HH''\coloneqq G\curvearrowright\ee_1\curvearrowright\dots\curvearrowright\ee_r$.  We now distinguish between five cases, see \cref{fig:DisjointEdgesMultiTriangle}:
  \begin{enumerate}
  \item $e_{m+1}$ and $e_{m+2}$ are attached to the two vertices of a single isolated multiedge of $\HH'$,
  \item $e_{m+1}$ and $e_{m+2}$ are attached to two distinct isolated multiedges of $\HH'$
  \item $e_{m+1}$ and $e_{m+2}$ are attached to an isolated multiedge and the multitriangle of $\HH'$, respectively,
  \item $e_{m+1}$ and $e_{m+2}$ are attached to two distinct vertices in the multitriangle of $\HH'$,
  \item $e_{m+1}$ and $e_{m+2}$ form a multiedge in $\HH''$.
  \end{enumerate}
  In Cases (1) and (4), $\HH\coloneqq\HH''\curvearrowright e_{m+1}$ yields an excised triangle.
  In Case (5), $\HH\coloneqq\HH''\curvearrowright \{e_{m+1},e_{m+2}\}$ yields an excised triangle.
  In Cases (2) and (3), let $\ee\in\EE(\HH')$ an isolated multiedge of $\HH'$ that either $e_{m+1}$ or $e_{m+2}$ are connected to.  Then $\HH\coloneqq\HH''\curvearrowright e_{m+1}\curvearrowright e_{m+2}\curvearrowright \ee$ is an excised triangle.
\end{proof}

\begin{figure}[t]
  \centering
  \begin{tikzpicture}
    \node (first)
    {
      \begin{tikzpicture}[x={(0.7,0)},y={(0,0.7)}]
        \node[circle, scale= 0.4, fill](v1) at (0,0){};
        \node [circle, scale= 0.4, fill](v2) at (0,3){};
        \node [circle, scale= 0.4, fill](v3) at (3,0){};
        \node [circle, scale= 0.4, fill](v4) at (3,3){};
        \node [circle, scale= 0.4, fill= Green](v5) at (1.5,2) [label={below, text = Green}: \tiny \circled{$2$}] {};
        \node [circle, scale= 0.4, fill=Green](v6) at (-1,-1) [label={below, text = Green}: \tiny \circled{$1$}] {};
        \node[circle, scale= 0.4, fill](v7) at (5,0){};
        \node[circle, scale= 0.4, fill](v8) at (5,3){};
        \node[circle, scale= 0.4, fill=Green](v9) at (6,-1)[label={below, text = Green}: \tiny \circled{$3$}] {};
        \node[circle, scale= 0.4, fill](v10) at (7,0){};
        \node[circle, scale= 0.4, fill](v11) at (9,0){};
        \node[circle, scale= 0.4, fill](v12) at (8,3){};
        \node[circle, scale= 0.4, fill= Green](v13) at (10,-1)[label={below, text = Green}: \tiny \circled{$5$}] {};
        \node[circle, scale= 0.4  , fill= Green](v14) at (10,2)[label={right, text = Green}: \tiny \circled{$4$}] {};
        \node[draw=none] (ellipsis1) at (4,1.5) {$\cdots$};
        

        \draw (v1)--(v2);
        \draw (v3) --(v4);
        \draw (v2) to [bend left=30] (v5) [Green] node[left,  yshift = 1mm] {\tiny $e_{m+2}$};
        \draw (v4) to [bend right=30] (v5) [Green] node[right,  yshift = 1mm, xshift = 0.5mm] {\tiny $e_{m+1}$};
        \draw (v2) to [bend right = 15] node[above left, color=Green]{\tiny $e_{m+1}$} (v6) [Green];
        \draw  (v1) to [bend right=15] (v6)[Green] node[midway, right, below,yshift = -2mm] {\tiny $e_{m+2}$}; 
        \draw (v7)--(v8);
        \draw (v7) to [bend left=15](v9) [Green] node[left, yshift = 1mm] {\tiny $e_{m+1}$};
        \draw (v10) to [bend right=15](v9) [Green] node[right, yshift = 1mm , xshift = 1mm ] {\tiny $e_{m+2}$};
        \draw (v10)--(v11);
        \draw (v10)--(v12);
        \draw (v11)--(v12);
        \draw (v12) -- (v14) [Green] node[pos = 0.6, above, yshift = 1mm ] {\tiny $e_{m+1}$};
        \draw (v11)--(v14) [Green]node[pos = 0.4, right ] {\tiny $e_{m+2}$};
        \draw (v13) to [bend right = 15] (v11) [Green] node[right, Green, xshift = 3mm, yshift = -2mm] {\tiny $e_{m+1}$};
        \draw (v11) to [bend right=15](v13) [Green] node[left, Green, yshift = 2mm, xshift = -3mm] {\tiny $e_{m+2}$};
      \end{tikzpicture}
    };
  \end{tikzpicture}
  \caption{The excised triangle $\HH'$ and the five cases of $\HH''$\pointOrNoPoint}
  \label{fig:DisjointEdgesMultiTriangle}
\end{figure}
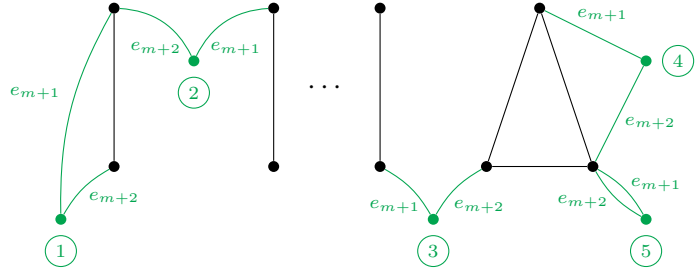


%% file: galaxy.tex
\section{Tropical galaxies of Laman graphs}\label{sec:galaxy}
In this section, we introduce the main object of study for this paper as well as some of its basic properties.

\begin{definition}
  \label{def:tropicalGalaxy}
  Let $G$ be a Laman graph. The \emph{tropical galaxy} of $G$ is a directed acyclic graph $\Gamma_G$ whose vertices are all possible excisions of $G$ and whose edges represent excisions:
  \begin{itemize}
  \item $V(\Gamma_G) = \{ G\curvearrowright \ee_1 \curvearrowright \dots \curvearrowright \ee_r\mid (\ee_1,\dots,\ee_r) \text{ is a valid chain of excisions} \}$, and
  \item $(\HH,\HH')\in E(\Gamma_G)$ if and only if $\HH'=\HH\curvearrowright \ee$ for some multiedge $\ee\in\EE(\HH)$.
  \end{itemize}
  Note that $G\in V(\Gamma_G)$ is the unique source in $\Gamma_G$.
\end{definition}

\begin{definition}
  \label{def:tropicalPairing}
  On the vertices of $\Gamma_G$ we have a \emph{galactic pairing}
  \begin{equation*}
    V(\Gamma_G) \times V(\Gamma_G) \longrightarrow \ZZ_{\geq 0},\qquad (\HH,\HH')\longmapsto \langle\HH,\HH'\rangle\coloneqq \Trop(\HH)\cdot (-\Trop(\HH')).
  \end{equation*}
\end{definition}

Here are some basic properties of the tropical intersection pairing on $\Gamma_G$:
\begin{lemma}
  \label{lem:tropicalPairing}\
  Let $G$ be a Laman graph.  Then the galactic pairing is
  \begin{enumerate}
  \item symmetric, i.e., $\langle\HH,\HH'\rangle = \langle\HH',\HH\rangle$ for all $\HH,\HH'\in V(\Gamma_G)$.
  \item monotonic, i.e., $\langle\HH,\HH'\rangle \geq \langle\HH,\HH''\rangle$ for all $\HH\in V(\Gamma_G)$ and $(\HH',\HH'')\in E(\Gamma_G)$.
  \item trivial on branches, i.e., $\langle\HH,\HH'\rangle = 0$ for all $\HH,\HH'\in V(\Gamma_G)$ if there exists some $\GG\in V(\Gamma_G)\setminus\{G\}$ that is connected to $\HH$ and $\HH'$ via a sequence of edges.
  \item binary on the leaves, i.e., $\langle\HH,\FF\rangle\in\{0,1\}$ for all $\HH,\FF \in V(\Gamma_G)$, $\FF$ leaf.
  \end{enumerate}
\end{lemma}
\begin{proof}\
  \begin{enumerate}
  \item Follows from $\Trop(\HH)\wedge -\Trop(\HH') = -(-\Trop(\HH) \wedge \Trop(\HH'))$.
  \item Follows from \cref{lem:intersectionProductAndStar} and \cref{lem:excisionToStar}.
  \item Follows from \cref{cor:tropicalization} and \cite[Theorem 3.6.10]{MaclaganSturmfels2015}.
  \item Follows from the fact that $\Trop(\FF)=-\Trop(\FF)$, which makes the stable intersection $\Trop(\HH)\wedge(-\Trop(\FF)) = \Trop(\HH)\wedge \Trop(\FF)$ a tropical linear space. \qedhere
  \end{enumerate}
\end{proof}

Moreover, the pairing takes on the following concrete values for certain special cases:

\begin{lemma}
  \label{lem:chainGalaxy}
  Fix $\GG\in V(\Gamma_G)$ and let $\HH\in V(\Gamma_G)$ be an excised chain of multiedges.  Then either
  \begin{equation*}
    \langle \GG,\HH'\rangle = 0 \qquad \text{for all descendants } \HH' \text{ of } \HH
  \end{equation*}
  or
  \begin{equation*}
    \langle \GG,\HH'\rangle = 1 \qquad \text{for all descendants } \HH' \text{ of } \HH
  \end{equation*}
\end{lemma}
\begin{proof}
  Follows from $|\Trop(\HH)|=|\Trop(\HH')|$ for all descendants $\HH'$ of $\HH$.
\end{proof}

\begin{lemma}
  \label{lem:pairingTriangle}
  Fix $\GG\in V(\Gamma_G)$ and let $\HH\in V(\Gamma_G)$ be an excised triangle.  Then:
  \begin{enumerate}
  \item If there is no edge $(\HH,\FF)\in E(\Gamma_G)$ with $\langle \GG,\FF\rangle > 0$, then $\langle \GG,\HH\rangle = 0$.
  \item If there is exactly one edge $(\HH,\FF)\in E(\Gamma_G)$ with $\langle \GG,\FF\rangle > 0$, then $\langle \GG,\HH\rangle = \langle \GG,\FF\rangle$.
  \end{enumerate}
\end{lemma}

\begin{proof}\
  \begin{enumerate}
  \item Note that $|\Trop(\HH)|\subseteq |\Trop(\FF_1)|\cup|\Trop(\FF_2)|\cup|\Trop(\FF_3)|$.  And $\langle \GG,\FF_i\rangle = 0$ means that for $u\in\RR^m$ generic
    \begin{equation*}
      |\Trop(\GG)|\cap |-\Trop(\FF_i)+u| = \emptyset
    \end{equation*}
    and hence $|\Trop(\GG)|\cap|-\Trop(\HH)+u| = \emptyset$.
  \item As in $(1)$, $|\Trop(\HH)|\subseteq |\Trop(\FF_1)|\cup|\Trop(\FF_2)|\cup|\Trop(\FF_3)|$. If $\langle \GG,\FF_i\rangle > 0$ for some $i \in \{1,2,3\}$ and zero on the other leaves, then, since $|\Trop(\GG)|\cap|-\Trop(\HH)+u|\subseteq |\Trop(\GG)|\cap |-\Trop(\FF_i)+u|$ and, by monotonicity $\langle\GG,\HH\rangle \geq \langle\GG,\FF_i\rangle$, it follows that $\langle \GG,\HH\rangle = \langle \GG,\FF_i\rangle$.
  \qedhere
  \end{enumerate}
\end{proof}


%% file: subadditivity.tex
\section{Subadditivity}\label{sec:subadditivity}
\cref{sec:galaxy} closed with a couple of results on the value of the galactic pairing on the outer stars, i.e., excised triangles and fully excised graphs.  To understand $\langle G, G\rangle=2\cdot c_2(G)$, it is therefore important to understand the subadditivity of the galactic pairing in the following sense:

\begin{proposition}
  \label{prop:subadditivity}
  Let $\GG,\GG'\in V(\Gamma_G)$ be two tropical stars.  Let $d(\cdot,\cdot)$ denote the distance on the directed graph $\Gamma_G$.  Then for all $\ell>0$ so that there is some $\HH\in V(\Gamma_G)$ with $d(\GG',\HH)=\ell$, we have
  \begin{equation*}
    \langle \GG, \GG'\rangle \leq \sum_{\substack{\HH\in V(\Gamma_G)\\ d(\GG',\HH)=\ell}} \langle \GG,\HH\rangle.
  \end{equation*}
\end{proposition}
\begin{proof}
  Using \cref{lem:excisionToStar}, we obtain $|\Trop(\GG')|\subseteq \bigcup_{\HH\in V(\Gamma_G), d(\GG',\HH)=\ell} |\Trop(\HH)|$, from which the statement follows as in the proof \cref{lem:pairingTriangle}.
\end{proof}

First note that the ``sub'' in ``subadditive'' is necessary, as one can find examples which show that the galactic pairing is not additive:

\begin{example}
  \label{ex:nonAdditivity}
  Let $G$ be a chain of triangles with 11 edges labelled as in \cref{fig:excisions2}.  Consider its two excisions $\GG\coloneqq G\curvearrowright 2\curvearrowright 4\curvearrowright 6$ and $\HH \curvearrowright 1\curvearrowright 23\curvearrowright 45\curvearrowright 10$.
  Let further $\FF_i\coloneqq \HH \curvearrowright \ee_i$ where $\ee_1,\ee_2,\ee_3$ denote the multiedges in the multitriangle of $\HH$. Then one can compute that
  \begin{equation*}
    \langle \GG, \HH \rangle = 2\neq 3=\underbrace{\langle \GG,\FF_1\rangle}_{=1} + \underbrace{\langle \GG,\FF_2\rangle}_{=1} + \underbrace{\langle \GG,\FF_3\rangle}_{=1}
  \end{equation*}
  which shows that the galactic pairing is not additive.
\end{example}

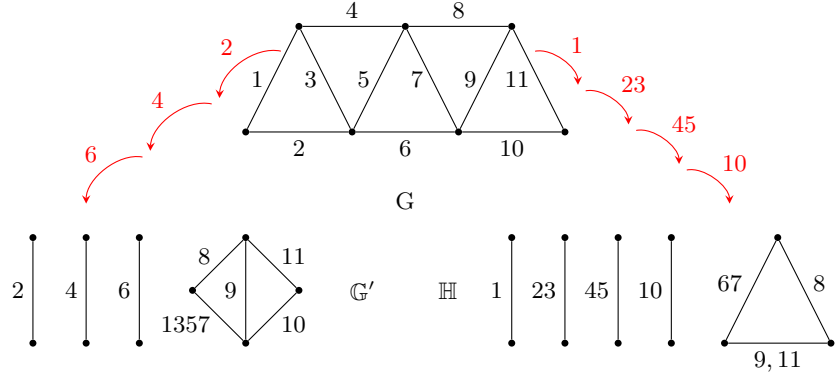
\begin{figure}
  \centering
  \begin{tikzpicture}[every node/.style = {font=\footnotesize}]

    \node(first)
    {
      \begin{tikzpicture}[x={(0.7,0)},y={(0,0.7)}]
        \node[circle, scale= 0.3, fill](a1) at (4,4)[]{};
        \node[circle, scale= 0.3, fill](a2) at (6,4)[]{};
        \node[circle, scale= 0.3, fill](a3) at (5,6)[]{};
        \node[circle, scale= 0.3, fill](a4) at (7,6)[]{};
        \node[circle, scale= 0.3, fill](a5) at (8,4)[]{};
        \node[circle, scale= 0.3, fill](a6) at (9,6)[]{};
        \node[circle, scale= 0.3, fill](a7) at (10,4)[]{};

        \draw (a1)--(a2) node[midway, below] {$2$};
        \draw (a2)--(a3) node [midway,left] {$3$};
        \draw (a3)--(a1) node[midway, left] {$1$};
        \draw (a4)--(a2) node[midway, left] {$5$};
        \draw (a4)--(a3) node[midway, above] {$4$};
        \draw (a4)--(a5) node[midway, left] {$7$};
        \draw (a2)--(a5) node[midway, below] {$6$};
        \draw (a6)--(a5) node[midway, left] {$9$};
        \draw (a6)--(a4) node[midway, above] {$8$};
        \draw (a6)--(a7) node[midway, left] {$11$};
        \draw (a5)--(a7) node[midway, below] {$10$};

        \node[circle, scale= 0.3, fill](b1) at (0,0)[]{};
        \node[circle, scale= 0.3, fill](b2) at (0,2)[]{};
        \node[circle, scale= 0.3, fill](b3) at (1,0)[]{};
        \node[circle, scale= 0.3, fill](b4) at (1,2)[]{};
        \node[circle, scale= 0.3, fill](b5) at (2,0)[]{};
        \node[circle, scale= 0.3, fill](b6) at (2,2)[]{};
        \node[circle, scale= 0.3, fill](b7) at (3,1)[]{};
        \node[circle, scale= 0.3, fill](b8) at (4,0)[]{};
        \node[circle, scale= 0.3, fill](b9) at (4,2)[]{};
        \node[circle, scale= 0.3, fill](b10) at (5,1)[]{};

        \draw (b1)--(b2) node[midway, left] {$2$};
        \draw (b3)--(b4) node[midway, left] {$4$};
        \draw (b5)--(b6) node[midway, left] {$6$};
        \draw (b7)--(b8) node[midway, left, yshift = -1mm] {$1357$};
        \draw (b7)--(b9) node[midway, left, yshift = 1mm] {$8$};
        \draw (b8)--(b9) node[midway, left] {$9$};
        \draw (b9)--(b10) node[midway, right, yshift = 1mm] {$11$};
        \draw (b8)--(b10) node[midway, right, yshift = -1mm] {$10$};

        \node[circle, scale= 0.3, fill](c1) at (9,0)[]{};
        \node[circle, scale= 0.3, fill](c2) at (9,2)[]{};
        \node[circle, scale= 0.3, fill](c3) at (10,0)[]{};
        \node[circle, scale= 0.3, fill](c4) at (10,2)[]{};
        \node[circle, scale= 0.3, fill](c5) at (11,0)[]{};
        \node[circle, scale= 0.3, fill](c6) at (11,2)[]{};
        \node[circle, scale= 0.3, fill](c7) at (12,0)[]{};
        \node[circle, scale= 0.3, fill](c8) at (12,2)[]{};
        \node[circle, scale= 0.3, fill](c9) at (13,0)[]{};
        \node[circle, scale= 0.3, fill](c10) at (14,2)[]{};
        \node[circle, scale= 0.3, fill](c11) at (15,0)[]{};

        \draw (c1)--(c2) node[midway, left] {$1$};
        \draw (c3)--(c4) node[midway, left] {$23$};
        \draw (c5)--(c6) node[midway, left] {$45$};
        \draw (c7)--(c8) node[midway, left] {$10$};
        \draw (c9)--(c10) node[midway, left, yshift = 1mm] {$67$};
        \draw (c9)--(c11) node[midway, below] {$9, 11$};
        \draw (c10)--(c11) node[midway, right, yshift = 1mm] {$8$};

        \draw [-stealth,shorten <=1mm, shorten >=1mm] (4.8,5.5) to [bend right = 50] node[above left,  color=red]{$2$} (3.5,4.5) [red];
        \draw [-stealth,shorten <=1mm, shorten >=1mm] (3.5,4.5) to [bend right = 50] node[above left,  color=red]{$4$} (2.2,3.5) [red];
        \draw [-stealth,shorten <=1mm, shorten >=1mm] (2.2,3.5) to [bend right = 50] node[above left,  color=red]{$6$} (1,2.5) [red];
        \draw [-stealth,shorten <=1mm, shorten >=1mm] (9.3,5.5) to [bend left = 50] node[above right,  color=red]{$1$} (10.25,4.75) [red];
        \draw [-stealth,shorten <=1mm, shorten >=1mm] (10.25,4.75) to [bend left = 50] node[above right,  color=red]{$23$} (11.2,4) [red];
        \draw [-stealth,shorten <=1mm, shorten >=1mm] (11.2,4) to [bend left = 50] node[above right,  color=red]{$45$} (12.15,3.25) [red];
        \draw [-stealth,shorten <=1mm, shorten >=1mm] (12.15,3.25) to [bend left = 50] node[above right,  color=red]{$10$} (13.1,2.5) [red];

        \node (d1) at (7,2.7) []{G};
        \node (d2) at (6.2,1) []{$\GG'$};
        \node (d3) at (7.8,1) []{$\HH$};
      \end{tikzpicture}
    };
  \end{tikzpicture}
  \caption{The Laman graph $G$ and its excisions $\GG'$, $\HH$ from \cref{ex:nonAdditivity}\pointOrNoPoint}
  \label{fig:excisions2}
\end{figure}

Hence the main question is to find conditions under which the galactic pairing actually increases in the sense of
\begin{equation*}
  \langle \GG, \GG'\rangle > \max\Big(\big\{ \langle \GG,\HH\rangle \mid \HH\in V(\Gamma_G), d(\GG',\HH)=1\big\}\Big).
\end{equation*}
While this seems to be a purely combinatorial question at first glance, it is difficult to translate the geometric nature of the tropical intersection product purely into combinatorics on the graphs.  We therefore close the section with an unfortunately not purely combinatorial condition under which the pairing does increase:

\begin{definition}
  \label{def:additivePair}
  Let $\HH,\HH'\in V(\Gamma_G)$ be two excised triangles of $G$, with triangle multiedges $\aa$, $\bb$, $\cc$, and $\aa'$, $\bb'$, $\cc'$ respectively. For $i=1,2$, consider $\FF_i=\HH\curvearrowright\ee_i$ and $\FF_i'=\HH\curvearrowright\ee_i'$ for some $\ee_i\in \{\aa, \bb, \cc\}$ and $\ee_i'\in\{\aa', \bb', \cc'\}$ such that $(\FF_1,\FF_1') \neq (\FF_2,\FF_2')$ and $\langle \FF_i,\FF_i'\rangle=1$.  We say that the two pairs $(\FF_1,\FF_1')$ and $(\FF_2,\FF_2')$ are \emph{additive} if there is some $w\in\RR^m$ such that
  \begin{equation*}
    u_1 - u_1' = w = u_2 - u_2' \text{ where } (u_1, u_1') \neq (u_2, u_2')
  \end{equation*}
  and $u_i$ and $u_i'$ induce chains of flats on $\HH$ and $\HH'$ where $\ee_i$ and $\ee_i'$ come before $\{\aa, \bb, \cc\} \setminus \{\ee_i\}$ and $\{\aa', \bb', \cc'\} \setminus \{\ee_i'\}$ respectively, i.e.,
  \begin{align*}
    &u_{i,j}>u_{i,k} \text{ for } j\in \ee, k\in \{\aa, \bb, \cc\} \setminus \{\ee\} \quad \text{and}\\
    &u_{i,j}'>u_{i,k}' \text{ for } j\in \ee', k\in \{\aa', \bb', \cc'\} \setminus \{\ee'\}.
  \end{align*}
\end{definition}

\begin{lemma}
  \label{lem:coneShift}
  Let $\sigma$ and $\sigma'$ be two cones in $\RR^m$.  Then for $w\in\RR^m$ we have
  \begin{equation*}
    \sigma \cap (w + \sigma') \neq \emptyset \quad \iff \quad w \in \sigma + (-\sigma').
  \end{equation*}
\end{lemma}
\begin{proof}
  The ``$\Rightarrow$'' direction is straightforward.  For the ``$\Leftarrow$'' direction, consider a translation $w\in\RR^m$ such that $\sigma \cap (w + \sigma') \neq \emptyset$, say $u\in \sigma \cap (w + \sigma')$ and set $u'\coloneqq w-u$.  Then $u\in w+\sigma'$ implies $u'\in -\sigma'$ and hence $w=u+u'\in\sigma+(-\sigma')$.
\end{proof}

\begin{theorem}
  \label{thm:additive}
  Let $G$ be a Laman graph and $\HH,\HH'\in V(\Gamma_G)$ be two of its excised triangles. Suppose for $i=1,2$ there are $\FF_i=\HH\curvearrowright\ee_i$ and $\FF_i'=\HH\curvearrowright\ee_i'$ for some $\ee_i\in \EE(\HH)$ and $\ee_i'\in\EE(\HH')$ such that $\langle \FF_i,\FF_i'\rangle=1$.  Then
  \begin{equation*}
    (\FF_1,\FF_1') \text{ and } (\FF_2,\FF_2') \text{ additive} \quad \Longrightarrow \quad \langle \HH,\HH'\rangle = 2.
  \end{equation*}
\end{theorem}
\begin{proof}
  We first show that $\langle \HH,\HH'\rangle\geq 2$.
  Suppose that $(\FF_1,\FF_1')$ and $(\FF_2,\FF_2')$ are additive, i.e., there is some $w\in\RR^m$ such that $u_1 - u_1' = w = u_2 - u_2'$ with $u_i$ and $u_i'$ satisfying the properties in \cref{def:additivePair}.

  Using the notation in \cref{cor:tropicalization} (3), which states that $\Trop(\HH)$ and $\Trop(\HH')$ each consists of three maximal cones induced by the triangle multiedges, we have $u_i\in\sigma_{\ee_i}\in\Trop(\HH)$ and $u_i'\in\sigma_{\ee_i'}\in\Trop(\HH')$.  Hence $u_i=w-u_i'\in (-\Trop(\HH')+w)$, and combining both we obtain $u_i\in \sigma_{\ee_i}\cap(\sigma_{\ee_i'}+w)$ by \cref{lem:coneShift}.  Moreover, $\langle \FF_i,\FF_i'\rangle=1$ implies that $\sigma_{\ee_i}\cap(\sigma_{\ee_i'}+w)\in \Trop(\HH)\wedge(-\Trop(\HH')+w)$, which shows that $\Trop(\HH)\wedge(-\Trop(\HH')+w)$ consists of at least two polyhedra.  Hence $\langle \HH,\HH'\rangle\geq 2$.

  Next we show that $\langle \HH,\HH'\rangle\leq 2$.  For that, consider a $3$-dimensional coordinate subspace $V$ that is a complement to $\mathrm{Lineality}(\Trop(\HH))+\mathrm{Lineality}(\Trop(\HH'))$, where $\mathrm{Lineality}(\cdot)$ denotes the lineality space.  On $V\cong\RR^3$, $\Trop(\HH)\cap V$ and $\Trop(\HH')\cap V$ remain tropical linear spaces and $(\Trop(\HH)\cap V)\cdot (-\Trop(\HH\cap V))=\Trop(\HH)\cdot(-\Trop(\HH'))$.  Therefore, $\Trop(\HH)\cdot(-\Trop(\HH'))$ must be either $0$, $1$, or $2$.
\end{proof}


%% file: arborealPairs.tex
\section{Arboreal Pairs}\label{sec:arboreal}
In this section, we show how fully excised graphs that are positively paired are equivalent to arboreal pairs in \cite{Ardila-MantillaEurPenaguiao2024}.

\begin{definition}
  \label{def:arboreal}
  Given two maximal (reduced) chains of flats $F_\bullet$ and $F_\bullet'$ of a Laman graph $G$, the \emph{intersection graph} of $F_\bullet$ and $F_\bullet'$ is the bipartite multigraph $I_{F_\bullet,F_\bullet'}$ with the following vertex and edge sets:
  \begin{align*}
    V(I_{F_\bullet,F_\bullet'})&\coloneqq\{F_j\mid j=1,\dots,r\}\sqcup\{F_j'\mid j=1,\dots,r\},\\
    E(I_{F_\bullet,F_\bullet'})&\coloneqq\{ (F_{j_i},F_{k_i}') \mid i=1,\dots,m\},
  \end{align*}
  where $F_{j_i}$ and $F_{j_i}'$ are the unique reduced flats with $i\in F_{j_i}$ and $i\in F_{j_i}'$.
  We say $F_\bullet$ and $F_\bullet'$ form an \emph{arboreal pair}, if the intersection graph $I_{F_\bullet,F_\bullet'}$ is a tree.
\end{definition}

\begin{example}
  \label{ex:arboreal}
  \cref{fig:Intersection} shows a Laman graph G and its intersection graph for the maximal chains of flats below.
  \begin{equation*}
    \begin{array}{rccccccccc}
      F_1 \colon &\emptyset&\subsetneq &1&\subsetneq& 14 &\subsetneq& 12345&\subsetneq& [7]\\
      F_2 \colon &\emptyset&\subsetneq &2&\subsetneq& 26 &\subsetneq& 1236 &\subsetneq& [7]\\
      F_3 \colon &\emptyset&\subsetneq &1&\subsetneq& 123&\subsetneq& 12345&\subsetneq& [7]
    \end{array}
  \end{equation*}
  While $I_{F_1,F_2}$ is a tree, $I_{F_2,F_3}$ contains a multiedge.  Hence $F_1,F_2$ form an arboreal pair, and $F_2, F_3$ do not.
\end{example}

\begin{figure}[t]
  \centering
  \begin{tikzpicture}
    \node (first)
    {
      \begin{tikzpicture}[x={(0.7,0)},y={(0,0.7)}]
        \node [circle, scale = 0.3, fill](v1) at (0,0){};
        \node [circle, scale = 0.3, fill](v2) at (2,0){};
        \node [circle, scale = 0.3, fill](v3) at (4,0){};
        \node [circle, scale = 0.3, fill](v4) at (1,2){};
        \node [circle, scale = 0.3, fill](v5) at (3,2){};

        \draw (v1) -- (v4)node[midway,left, blue] {1}[blue];
        \draw (v1) -- (v2)node[midway,below, Orange] {2} [Orange];
        \draw (v2) -- (v4)node[midway,left, Orange] {3} [Orange];
        \draw (v4) -- (v5)node[midway,above, RedViolet] {4} [RedViolet];
        \draw (v2) -- (v5)node[midway,right, Orange] {5}[Orange];
        \draw (v5) -- (v3)node[midway,right, Green] {7}[Green];
        \draw (v3) -- (v2)node[midway,below, Green] {6} [Green];

      \end{tikzpicture}
    };
      \node[anchor=north] (first2) at (first.south)
    {
    \definecolor{red-violet}{rgb}{0.78, 0.08, 0.52}
      \begin{tikzpicture}[x={(0.7,0)},y={(0,0.7)}]
        \node [circle, scale = 0.3, fill](v1) at (0,0){};
        \node [circle, scale = 0.3, fill](v2) at (2,0){};
        \node [circle, scale = 0.3, fill](v3) at (4,0){};
        \node [circle, scale = 0.3, fill](v4) at (1,2){};
        \node [circle, scale = 0.3, fill](v5) at (3,2){};

        \draw (v1) -- (v4)node[midway,left, Orange] {1}[Orange];
        \draw (v1) -- (v2)node[midway,below, blue] {2} [blue];
        \draw (v2) -- (v4)node[midway,left, Orange] {3} [Orange];
        \draw (v4) -- (v5)node[midway,above, Green] {4} [Green];
        \draw (v2) -- (v5)node[midway,right, Green] {5}[Green];
        \draw (v5) -- (v3)node[midway,right, Green] {7}[Green];
        \draw (v3) -- (v2)node[midway,below, RedViolet] {6}[RedViolet];

      \end{tikzpicture}
    };
     \node[anchor=north] (first3) at (first2.south)
    {
      \begin{tikzpicture}[x={(0.7,0)},y={(0,0.7)}]
        \node [circle, scale = 0.3, fill](v1) at (0,0){};
        \node [circle, scale = 0.3, fill](v2) at (2,0){};
        \node [circle, scale = 0.3, fill](v3) at (4,0){};
        \node [circle, scale = 0.3, fill](v4) at (1,2){};
        \node [circle, scale = 0.3, fill](v5) at (3,2){};

        \draw (v1) -- (v4)node[midway,left, blue] {1}[blue];
        \draw (v1) -- (v2)node[midway,below, RedViolet] {2}[RedViolet];
        \draw (v2) -- (v4)node[midway,left, RedViolet] {3}[RedViolet];
        \draw (v4) -- (v5)node[midway,above, Orange] {4}[Orange];
        \draw (v2) -- (v5)node[midway,right, Orange] {5}[Orange];
        \draw (v5) -- (v3)node[midway,right, Green] {7}[Green];
        \draw (v3) -- (v2)node[midway,below, Green] {6}[Green];

      \end{tikzpicture}
    };
     \node[anchor=west, yshift = -10mm] (second) at (first.east)
    {
      \begin{tikzpicture}[x={(0.7,0)},y={(0,0.7)}]
        \node [circle, scale= 0.3, fill](v1) at (0,0){} ;
        \node [circle, scale= 0.3, fill](v2) at (2,0){} ;
        \node [circle, scale= 0.3, fill](v3) at (4,0){} ;
        \node [circle, scale= 0.3, fill](v4) at (6,0){} ;
        \node [circle, scale= 0.3, fill](v5) at (0,2){} ;
        \node [circle, scale= 0.3, fill](v6) at (2,2){} ;
        \node [circle, scale= 0.3, fill](v7) at (4,2){} ;
        \node [circle, scale= 0.3, fill](v8) at (6,2){} ;

        \node(V1) at (0,0) [below, yshift = -0.7mm, blue]{2};
        \node(V2) at (2,0) [below, yshift = -0.7mm, RedViolet]{6};
        \node(V3) at (4,0) [below, yshift = -0.7mm, Orange]{13};
        \node(V4) at (6,0) [below, yshift = -0.7mm, Green]{457};
        \node(V5) at (0,2) [above, yshift = 0.7mm, blue]{1};
        \node(V6) at (2,2) [above, yshift = 0.7mm, RedViolet]{4};
        \node(V7) at (4,2) [above, yshift = 0.7mm, Orange]{235};
        \node(V8) at (6,2) [above, yshift = 0.7mm, Green]{67};
        
        \node(A1) at (0,2) [above, left, xshift = -0.7mm, yshift = 2.8mm]{$F_1^{\text{red}} \colon $};
        \node(A2) at (0,0) [below, left, xshift = -0.7mm, yshift = -4mm]{$F_2^{\text{red}} \colon $};

        \draw (v5)--(v3)--(v7)--(v4)--(v8)--(v2);
        \draw (v7)--(v1);
        \draw (v4)--(v6);
      \end{tikzpicture}
    };
      \node[anchor=north, yshift = -5mm] (third) at (second.south)
    {
      \begin{tikzpicture}[x={(0.7,0)},y={(0,0.7)}]
        \node [circle, scale= 0.3, fill](v1) at (0,0) {} ;
        \node [circle, scale= 0.3, fill](v2) at (2,0){} ;
        \node [circle, scale= 0.3, fill](v3) at (4,0){} ;
        \node [circle, scale= 0.3, fill](v4) at (6,0){} ;
        \node [circle, scale= 0.3, fill](v5) at (0,2){} ;
        \node [circle, scale= 0.3, fill](v6) at (2,2){} ;
        \node [circle, scale= 0.3, fill](v7) at (4,2){} ;
        \node [circle, scale= 0.3, fill](v8) at (6,2){} ;

        \node(V1) at (0,0) [below, yshift = -0.7mm, blue]{1};
        \node(V2) at (2,0) [below, yshift = -0.7mm, RedViolet]{23};
        \node(V3) at (4,0) [below, yshift = -0.7mm, Orange]{45};
        \node(V4) at (6,0) [below, yshift = -0.7mm, Green]{67};
        \node(V5) at (0,2) [above, yshift = 0.7mm, blue]{2};
        \node(V6) at (2,2) [above, yshift = 0.7mm, RedViolet]{6};
        \node(V7) at (4,2) [above, yshift = 0.7mm, Orange]{13};
        \node(V8) at (6,2) [above, yshift = 0.7mm, Green]{457};

        \node(A1) at (0,2) [above, left, xshift = -0.7mm, yshift = 2.8mm]{$F_2^{\text{red}} \colon $};
        \node(A2) at (0,0) [below, left, xshift = -0.7mm, yshift = -4mm]{$F_3^{\text{red}} \colon $};

        \draw (v4)--(v8)--(v3);
        \draw (v8) to [bend right = 15] (v3);
        \draw (v7)--(v1);
        \draw (v4)--(v6);
        \draw (v7)--(v2)--(v5); 
      \end{tikzpicture}
    };
         \node[anchor=west] (fourth) at (second.east)
    {
       \begin{tikzpicture}[x={(0.7,0)},y={(0,0.7)}]
        \node[circle, scale= 0.3, fill](v2) at (7,2){};
        \node[circle, scale= 0.3, fill](v3) at (6,0){};
        \node[circle, scale= 0.3, fill](v4) at (5,2){};
        \node[circle, scale= 0.3, fill](v5) at (4,0){};
        \node[circle, scale= 0.3, fill](v6) at (3,2){};
        \node[circle, scale= 0.3, fill](v7) at (2,0){};
        \node[circle, scale= 0.3, fill](v8) at (1,2){};
        \node[circle, scale= 0.3, fill](v9) at (8,0){};

        \node(V2) at (7,2) [above, yshift = 0.7mm, Green]{67};
        \node(V3) at (6,0) [below, yshift = -0.7mm, Green]{457};
        \node(V4) at (5,2) [above, yshift = 0.7mm, RedViolet]{4};
        \node(V5) at (4,0) [below, yshift = -0.7mm, blue]{2};
        \node(V6) at (3,2) [above, yshift = 0.7mm, Orange]{235};
        \node(V7) at (2,0) [below, yshift = -0.7mm, Orange]{13};
        \node(V8) at (1,2) [above, yshift = 0.7mm, blue]{1};
        \node(V9) at (8,0) [below, yshift = -0.7mm, RedViolet]{6};

        \draw (v5) -- node[sloped,anchor=center] {$<$} (v6) -- node[sloped,anchor=center] {$<$} (v7) -- node[sloped,anchor=center] {$<$} (v8);
        \draw (v3) -- node[sloped,anchor=center] {$<$} (v6) ;
        \draw (v4) -- node[sloped,anchor=center] {$>$} (v3) ;
        \draw (v3) -- node[sloped,anchor=center] {$<$} (v2) -- node[sloped,anchor=center] {$<$} (v9);

      \end{tikzpicture}
    
    };
     \node[anchor=west] (fifth) at (third.east)
    {
      \begin{tikzpicture}[x={(0.7,0)},y={(0,0.7)}]
        \node[circle, scale= 0.3, fill](v2) at (7,2){};
        \node[circle, scale= 0.3, fill](v3) at (6,0){};
        \node[circle, scale= 0.3, fill](v4) at (5,2){};
        \node[circle, scale= 0.3, fill](v5) at (4,0){};
        \node[circle, scale= 0.3, fill](v6) at (3,2){};
        \node[circle, scale= 0.3, fill](v7) at (2,0){};
        \node[circle, scale= 0.3, fill](v8) at (1,2){};
        \node[circle, scale= 0.3, fill](v9) at (8,0){};

        \node(V2) at (7,2) [above, yshift = 0.7mm, Green]{457};
        \node(V3) at (6,0) [below, yshift = -0.7mm, Green]{67};
        \node(V4) at (5,2) [above, yshift = 0.7mm, RedViolet]{6};
        \node(V5) at (4,0) [below, yshift = -0.7mm, blue]{1};
        \node(V6) at (3,2) [above, yshift = 0.7mm, Orange]{13};
        \node(V7) at (2,0) [below, yshift = -0.7mm, RedViolet]{23};
        \node(V8) at (1,2) [above, yshift = 0.7mm, blue]{2};
        \node(V9) at (8,0) [below, yshift = -0.7mm, Orange]{45};

        \draw (v5)--(v6)--(v7)--(v8);
        \draw (v2) to [bend left=20] node[sloped,anchor=center] {$<$} (v9) ;
        \draw (v2) to [bend right=20] node[sloped,anchor=center] {$<$} (v9) ;
        \draw (v4) -- node[sloped,anchor=center] {$>$} (v3) ;
        \draw (v3) -- node[sloped,anchor=center] {$<$} (v2);

      \end{tikzpicture}
    };
    \end{tikzpicture}
  \caption{Laman graph G with intersection graphs $I_{F_1,F_2}$ (tree) and $I_{F_2,F_3}$ (with cycle)\pointOrNoPoint}
  \label{fig:Intersection}
\end{figure}
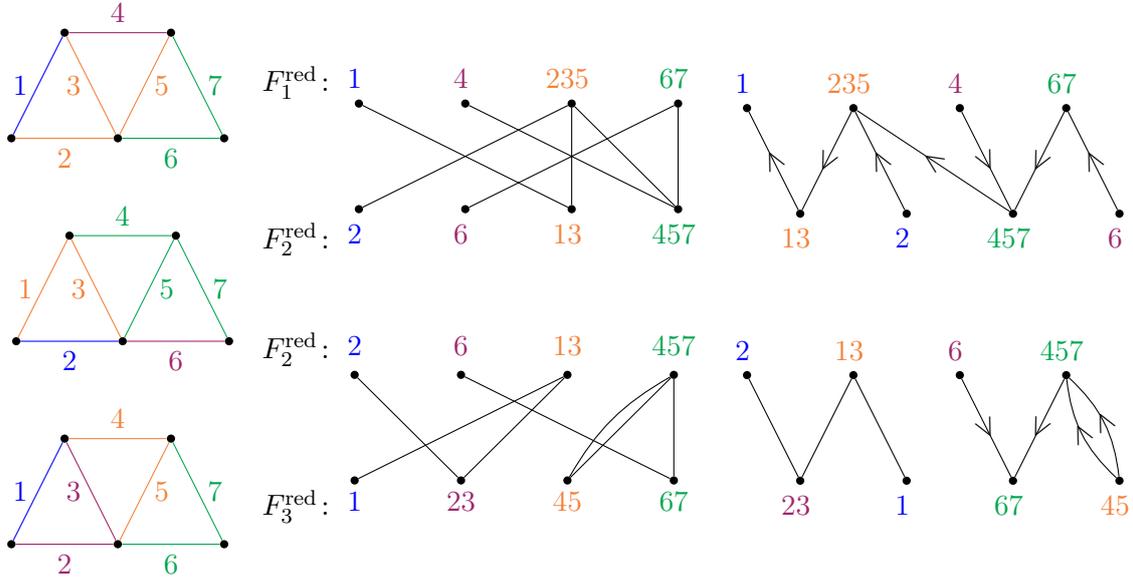

The statement of the following \cref{prop:arboreal} was communicated by Oliver Clarke and Ben Smith.  \cref{ex:columnReduction} illustrates the parts of the proof on an example.

\begin{proposition}
  \label{prop:arboreal}
  Let $F_\bullet, F_\bullet'$ be two maximal reduced chains of flats on $G$ giving rise to fully excised $\FF,\FF'\in V(\Gamma_G)$, respectively.  Then
  \begin{equation*}
    F \text{ and } F' \text{ form an arboreal pair}\quad\Longleftrightarrow\quad \langle\FF,\FF'\rangle=1.
  \end{equation*}
\end{proposition}
\begin{proof}
  For the ``$\Rightarrow$'' direction, suppose that $F_{\bullet}$ and $F_\bullet'$ form an arboreal pair, which means that $I_{F,F'}$ is a tree.
  Note that the tropical linear spaces $\Trop(\FF)$ and $\Trop(\FF')$ each consist of a single cell $\Span(\bbone_{\ee}\mid \ee\in\EE(\FF))$ and $\Span(\bbone_{\ee'}\mid \ee'\in\EE(\FF'))$, respectively.
  Let $r$ be the common length of the chains $F_{\bullet}$ and $F_\bullet'$.  As $G$ is a Laman graph, we have that $2r=m+1$.
  Let $A\in\RR^{m\times (m+1)}$ be the matrix whose whose columns are the indicator vectors of $F_j$ and $F_j'$, $j=1,\dots,r$.  We now show that $A$ is of full rank, which means that  $\Trop(\FF)$ and $\Trop(\FF')$ intersect transversally, which in turn implies $\langle\FF,\FF'\rangle=1$.

  Fix a root on the tree $I_{F,F'}$ and orient each edge of $I_{F,F'}$ so that it points towards the root. The resulting directed tree then gives us a sequence of column reductions which result in a matrix in reduced column echelon form. This matrix is clearly of full rank, see \cref{ex:columnReduction}.

  For the ``$\Leftarrow$'' direction, suppose that $I_{F_\bullet,F_\bullet'}$ has a cycle.  As $|V(I_{F_\bullet,F_\bullet'})|=m+1=|E(I_{F_\bullet,F_\bullet'})|+1$, it follows that $I_{F_\bullet,F_\bullet'}$ has at least two connected components.  Decompose $I_{F_\bullet,F_\bullet'}=I_1\sqcup I_2$ where each $I_i$ has at least one connected component, which in turn yields a decomposition $[m]=M_1\sqcup M_2$ where $M_i\coloneqq \bigcup_{F\in V(I_i)} F\subseteq [m]$.  Let $A\in\RR^{m\times (m+1)}$ be the matrix as before, but with the rows rearranged so that ground set elements in $M_1$ come before ground set elements in $M_2$ and with the columns rearranged so that flats in $V(I_1)$ come before flats in $V(I_2)$.  The result is a block-diagonal matrix with blocks $A_i\in\RR^{|M_i|\times|E(I_i)|}$, see \cref{ex:columnReduction}.  Note that at most one of the $A_i$ may be square. We now distinguish between two cases:

  If neither $A_1$ nor $A_2$ is square, then one of them must have more rows than columns, say $A_2$, which implies that its columns do not span the entire $\RR^{|M_2|}$, which in turn implies that the columns of $A$ do not span the entire $\RR^m$.

  If one of the $A_i$ is square, say $A_2$, then fixing an orientation on $I_2$ yields a linear combination of the columns that equals zero.  Hence the square submatrix has a non-trivial kernel, which means its columns do do not span the entire $\RR^{|M_2|}$, which in turn implies that the columns of $A$ do not span the entire $\RR^m$.

  In both cases, we have that $\dim(\Span(\bbone_{\ee}\mid \ee\in\EE(\FF))+\Span(\bbone_{\ee}\mid \ee'\in\EE(\FF')))\neq m$, which shows that $\Trop(\FF)\wedge \Trop(\FF')=\emptyset$, which in turn implies that $\langle \FF,\FF'\rangle=0$.
\end{proof}

\begin{example}
  \label{ex:columnReduction}
  Consider the Laman graph $G$ and the three chain of flats from \cref{ex:arboreal}.
  Picking the orientation on $I_{F_1,F_2}$ in \cref{fig:Intersection} gives the following sequence of column reductions for the indicator matrix $A\in\RR^{7\times8}$ which (modulo reordering of the columns) yields a column reduced echelon form that is clearly of full rank:
  \begin{center}
    \begin{tikzpicture}
      \node (A0)
      {
        \begin{tikzpicture}[add paren/.style={left delimiter={(},right delimiter={)}}]
          \matrix (m) [matrix of math nodes, row sep=-0.35mm, column sep=-0.35mm]
          { 
            1  & 0  & 0 & 0 & 0 & 0 & 1 & 0 \\
            0  & 0  & 1 & 0 & 1 & 0 & 0 & 0 \\
            0  & 0  & 1 & 0 & 0 & 0 & 1 & 0 \\
            0  & 1  & 0 & 0 & 0 & 0 & 0 & 1 \\
            0  & 0  & 1 & 0 & 0 & 0 & 0 & 1 \\
            0  & 0  & 0 & 1 & 0 & 1 & 0 & 0 \\
            0  & 0  & 0 & 1 & 0 & 0 & 0 & 1 \\ };
          \node[fit=(m-1-1) (m-7-8), add paren, inner sep=0pt] (submatrix) {};
          \node[above,font=\tiny,yshift=3mm,blue] at (m-1-1) {$1$};
          \node[above,font=\tiny,yshift=3mm,RedViolet] at (m-1-2) {$4$};
          \node[above,font=\tiny,yshift=3mm,Orange] at (m-1-3) {$235$};
          \node[above,font=\tiny,yshift=3mm,Green] at (m-1-4) {$67$};
          \node[above,font=\tiny,yshift=3mm,blue] at (m-1-5) {$2$};
          \node[above,font=\tiny,yshift=3mm,RedViolet] at (m-1-6) {$6$};
          \node[above,font=\tiny,yshift=3mm,Orange] at (m-1-7) {$13$};
          \node[above,font=\tiny,yshift=3mm,Green] at (m-1-8) {$457$};
        \end{tikzpicture}
      };
      \node[anchor=west,xshift=10mm] (A1) at (A0.east)
      {
        \begin{tikzpicture}[add paren/.style={left delimiter={(},right delimiter={)}},anchor=center]
          \matrix (m) [matrix of math nodes, row sep=-0.35mm, column sep=-0.35mm]
          { 
            1  & 0  & 0 & 0 & 0 & 0 & 1 & 0 \\
            0  & 0  & 1 & 0 & 1 & 0 & 0 & 0 \\
            0  & 0  & 1 & 0 & 0 & 0 & 1 & 0 \\
            0  & 1  & 0 & 0 & 0 & 0 & 0 & 1 \\
            0  & 0  & 1 & 0 & 0 & 0 & 0 & 1 \\
            0  & 0  & 0 & 0 & 0 & 1 & 0 & 0 \\
            0  & 0  & 0 & 1 & 0 & 0 & 0 & 1 \\ };
          \node[fit=(m-1-1) (m-7-8), add paren, inner sep=0pt] (submatrix) {};
          \node[above,font=\tiny,yshift=3mm,blue] at (m-1-1) {$1$};
          \node[above,font=\tiny,yshift=3mm,RedViolet] at (m-1-2) {$4$};
          \node[above,font=\tiny,yshift=3mm,Orange] at (m-1-3) {$235$};
          \node[above,font=\tiny,yshift=3mm,Green] at (m-1-4) {$67$};
          \node[above,font=\tiny,yshift=3mm,blue] at (m-1-5) {$2$};
          \node[above,font=\tiny,yshift=3mm,RedViolet] at (m-1-6) {$6$};
          \node[above,font=\tiny,yshift=3mm,Orange] at (m-1-7) {$13$};
          \node[above,font=\tiny,yshift=3mm,Green] at (m-1-8) {$457$};
        \end{tikzpicture}
      };
      \node[anchor=west,xshift=10mm] (A2) at (A1.east)
      {
        \begin{tikzpicture}[add paren/.style={left delimiter={(},right delimiter={)}},anchor=center]
          \matrix (m) [matrix of math nodes, row sep=-0.35mm, column sep=-0.35mm]
          { 
            1  & 0  & 0 & 0 & 0 & 0 & 1 & 0 \\
            0  & 0  & 1 & 0 & 1 & 0 & 0 & 0 \\
            0  & 0  & 1 & 0 & 0 & 0 & 1 & 0 \\
            0  & 1  & 0 & 0 & 0 & 0 & 0 & 0 \\
            0  & 0  & 1 & 0 & 0 & 0 & 0 & 1 \\
            0  & 0  & 0 & 0 & 0 & 1 & 0 & 0 \\
            0  & 0  & 0 & 1 & 0 & 0 & 0 & 0 \\ };
          \node[fit=(m-1-1) (m-7-8), add paren, inner sep=0pt] (submatrix) {};
          \node[above,font=\tiny,yshift=3mm,blue] at (m-1-1) {$1$};
          \node[above,font=\tiny,yshift=3mm,RedViolet] at (m-1-2) {$4$};
          \node[above,font=\tiny,yshift=3mm,Orange] at (m-1-3) {$235$};
          \node[above,font=\tiny,yshift=3mm,Green] at (m-1-4) {$67$};
          \node[above,font=\tiny,yshift=3mm,blue] at (m-1-5) {$2$};
          \node[above,font=\tiny,yshift=3mm,RedViolet] at (m-1-6) {$6$};
          \node[above,font=\tiny,yshift=3mm,Orange] at (m-1-7) {$13$};
          \node[above,font=\tiny,yshift=3mm,Green] at (m-1-8) {$457$};
        \end{tikzpicture}
      };
      \node[anchor=north,yshift=-5mm] (A3) at (A2.south)
      {
        \begin{tikzpicture}[add paren/.style={left delimiter={(},right delimiter={)}}]
          \matrix (m) [matrix of math nodes, row sep=-0.35mm, column sep=-0.35mm]
          { 
            1  & 0  & 0 & 0 & 0 & 0 & 1 & 0 \\
            0  & 0  & 0 & 0 & 1 & 0 & 0 & 0 \\
            0  & 0  & 1 & 0 & 0 & 0 & 1 & 0 \\
            0  & 1  & 0 & 0 & 0 & 0 & 0 & 0 \\
            0  & 0  & 0 & 0 & 0 & 0 & 0 & 1 \\
            0  & 0  & 0 & 0 & 0 & 1 & 0 & 0 \\
            0  & 0  & 0 & 1 & 0 & 0 & 0 & 0 \\ };
          \node[fit=(m-1-1) (m-7-8), add paren, inner sep=0pt] (submatrix) {};
          \node[below,font=\tiny,yshift=-3mm,blue] at (m-7-1) {$1$};
          \node[below,font=\tiny,yshift=-3mm,RedViolet] at (m-7-2) {$4$};
          \node[below,font=\tiny,yshift=-3mm,Orange] at (m-7-3) {$235$};
          \node[below,font=\tiny,yshift=-3mm,Green] at (m-7-4) {$67$};
          \node[below,font=\tiny,yshift=-3mm,blue] at (m-7-5) {$2$};
          \node[below,font=\tiny,yshift=-3mm,RedViolet] at (m-7-6) {$6$};
          \node[below,font=\tiny,yshift=-3mm,Orange] at (m-7-7) {$13$};
          \node[below,font=\tiny,yshift=-3mm,Green] at (m-7-8) {$457$};
        \end{tikzpicture}
      };
      \node[anchor=north,yshift=-5mm] (A4) at (A1.south)
      {
        \begin{tikzpicture}[add paren/.style={left delimiter={(},right delimiter={)}}]
          \matrix (m) [matrix of math nodes, row sep=-0.35mm, column sep=-0.35mm]
          { 
            1  & 0  & 0 & 0 & 0 & 0 & 1 & 0 \\
            0  & 0  & 0 & 0 & 1 & 0 & 0 & 0 \\
            0  & 0  & 1 & 0 & 0 & 0 & 0 & 0 \\
            0  & 1  & 0 & 0 & 0 & 0 & 0 & 0 \\
            0  & 0  & 0 & 0 & 0 & 0 & 0 & 1 \\
            0  & 0  & 0 & 0 & 0 & 1 & 0 & 0 \\
            0  & 0  & 0 & 1 & 0 & 0 & 0 & 0 \\ };
          \node[fit=(m-1-1) (m-7-8), add paren, inner sep=0pt] (submatrix) {};
          \node[below,font=\tiny,yshift=-3mm,blue] at (m-7-1) {$1$};
          \node[below,font=\tiny,yshift=-3mm,RedViolet] at (m-7-2) {$4$};
          \node[below,font=\tiny,yshift=-3mm,Orange] at (m-7-3) {$235$};
          \node[below,font=\tiny,yshift=-3mm,Green] at (m-7-4) {$67$};
          \node[below,font=\tiny,yshift=-3mm,blue] at (m-7-5) {$2$};
          \node[below,font=\tiny,yshift=-3mm,RedViolet] at (m-7-6) {$6$};
          \node[below,font=\tiny,yshift=-3mm,Orange] at (m-7-7) {$13$};
          \node[below,font=\tiny,yshift=-3mm,Green] at (m-7-8) {$457$};
        \end{tikzpicture}
      };
      \node[anchor=north,yshift=-5mm] (A5) at (A0.south)
      {
        \begin{tikzpicture}[add paren/.style={left delimiter={(},right delimiter={)}}]
          \matrix (m) [matrix of math nodes, row sep=-0.35mm, column sep=-0.35mm]
          { 
            0  & 0  & 0 & 0 & 0 & 0 & 1 & 0 \\
            0  & 0  & 0 & 0 & 1 & 0 & 0 & 0 \\
            0  & 0  & 1 & 0 & 0 & 0 & 0 & 0 \\
            0  & 1  & 0 & 0 & 0 & 0 & 0 & 0 \\
            0  & 0  & 0 & 0 & 0 & 0 & 0 & 1 \\
            0  & 0  & 0 & 0 & 0 & 1 & 0 & 0 \\
            0  & 0  & 0 & 1 & 0 & 0 & 0 & 0 \\ };
          \node[fit=(m-1-1) (m-7-8), add paren, inner sep=0pt] (submatrix) {};
          \node[below,font=\tiny,yshift=-3mm,blue] at (m-7-1) {$1$};
          \node[below,font=\tiny,yshift=-3mm,RedViolet] at (m-7-2) {$4$};
          \node[below,font=\tiny,yshift=-3mm,Orange] at (m-7-3) {$235$};
          \node[below,font=\tiny,yshift=-3mm,Green] at (m-7-4) {$67$};
          \node[below,font=\tiny,yshift=-3mm,blue] at (m-7-5) {$2$};
          \node[below,font=\tiny,yshift=-3mm,RedViolet] at (m-7-6) {$6$};
          \node[below,font=\tiny,yshift=-3mm,Orange] at (m-7-7) {$13$};
          \node[below,font=\tiny,yshift=-3mm,Green] at (m-7-8) {$457$};
        \end{tikzpicture}
      };
      \draw [->] (A0) -- node[above,font=\scriptsize] {$\textcolor{RedViolet}6\!\to\!\textcolor{Green}{67}$} (A1);
      \draw [->] (A1) -- node[above,font=\scriptsize,yshift=4mm] {$\textcolor{Green}{67}\!\to\!\textcolor{Green}{457}$} node[above,font=\scriptsize] {$\textcolor{RedViolet}{4}\!\to\!\textcolor{Green}{457}$} (A2);
      \draw [->] (A2) -- node[right,font=\scriptsize,yshift=2mm] {$\textcolor{blue}{2}\!\to\!\textcolor{Orange}{235}$}  node[right,font=\scriptsize,yshift=-2mm] {$\textcolor{Green}{457}\!\to\!\textcolor{Orange}{235}$} (A3);
      \draw [->] (A3) -- node[below,font=\scriptsize] {$\textcolor{Orange}{13}\!\leftarrow\!\textcolor{Orange}{235}$} (A4);
      \draw [->] (A4) -- node[below,font=\scriptsize] {$\textcolor{blue}{1}\!\leftarrow\!\textcolor{Orange}{13}$} (A5);
    \end{tikzpicture}
  \end{center}
  
  In contrast, picking the orientation on the connected component of $I_{F_2,F_3}$ from \cref{fig:Intersection} gives the following sequence of column reductions for a block of the indicator matrix which (modulo reordering of the columns) yields a column reduced echolon form that is clearly not of full rank:
  \begin{center}
    \begin{tikzpicture}
      \node (A0)
      {
        \begin{tikzpicture}[add paren/.style={left delimiter={(},right delimiter={)}}]
          \matrix (m) [matrix of math nodes, row sep=-0.35mm, column sep=-0.35mm]
          { 0 & 1 & 0 & 1 \\
            0 & 1 & 0 & 1 \\
            1 & 0 & 1 & 0 \\
            0 & 1 & 1 & 0 \\ };
          \node[fit=(m-1-1) (m-4-4), add paren, inner sep=0pt] (submatrix) {};
          \node[above,font=\tiny,yshift=3mm,RedViolet] at (m-1-1) {$6$};
          \node[above,font=\tiny,yshift=3mm,Green] at (m-1-2) {$457$};
          \node[above,font=\tiny,yshift=3mm,Green] at (m-1-3) {$67$};
          \node[above,font=\tiny,yshift=3mm,RedViolet] at (m-1-4) {$45$};
        \end{tikzpicture}
      };
      \node[anchor=west,xshift=10mm] (A1) at (A0.east)
      {
        \begin{tikzpicture}[add paren/.style={left delimiter={(},right delimiter={)}},anchor=center]
          \matrix (m) [matrix of math nodes, row sep=-0.35mm, column sep=-0.35mm]
          { 0 & 0 & 0 & 1 \\
            0 & 0 & 0 & 1 \\
            1 & 0 & 1 & 0 \\
            0 & 1 & 1 & 0 \\ };
          \node[fit=(m-1-1) (m-4-4), add paren, inner sep=0pt] (submatrix) {};
          \node[above,font=\tiny,yshift=3mm,RedViolet] at (m-1-1) {$6$};
          \node[above,font=\tiny,yshift=3mm,Green] at (m-1-2) {$457$};
          \node[above,font=\tiny,yshift=3mm,Green] at (m-1-3) {$67$};
          \node[above,font=\tiny,yshift=3mm,RedViolet] at (m-1-4) {$45$};
        \end{tikzpicture}
      };
      \node[anchor=west,xshift=10mm] (A2) at (A1.east)
      {
        \begin{tikzpicture}[add paren/.style={left delimiter={(},right delimiter={)}},anchor=center]
          \matrix (m) [matrix of math nodes, row sep=-0.35mm, column sep=-0.35mm]
          { 0 & 0 & 0 & 1 \\
            0 & 0 & 0 & 1 \\
            1 & 0 & 1 & 0 \\
            0 & 1 & 0 & 0 \\ };
          \node[fit=(m-1-1) (m-4-4), add paren, inner sep=0pt] (submatrix) {};
          \node[above,font=\tiny,yshift=3mm,RedViolet] at (m-1-1) {$6$};
          \node[above,font=\tiny,yshift=3mm,Green] at (m-1-2) {$457$};
          \node[above,font=\tiny,yshift=3mm,Green] at (m-1-3) {$67$};
          \node[above,font=\tiny,yshift=3mm,RedViolet] at (m-1-4) {$45$};
        \end{tikzpicture}
      };
      \node[anchor=west,xshift=10mm] (A3) at (A2.east)
      {
        \begin{tikzpicture}[add paren/.style={left delimiter={(},right delimiter={)}},anchor=center]
          \matrix (m) [matrix of math nodes, row sep=-0.35mm, column sep=-0.35mm]
          { 0 & 0 & 0 & 1 \\
            0 & 0 & 0 & 1 \\
            0 & 0 & 1 & 0 \\
            0 & 1 & 0 & 0 \\ };
          \node[fit=(m-1-1) (m-4-4), add paren, inner sep=0pt] (submatrix) {};
          \node[above,font=\tiny,yshift=3mm,RedViolet] at (m-1-1) {$6$};
          \node[above,font=\tiny,yshift=3mm,Green] at (m-1-2) {$457$};
          \node[above,font=\tiny,yshift=3mm,Green] at (m-1-3) {$67$};
          \node[above,font=\tiny,yshift=3mm,RedViolet] at (m-1-4) {$45$};
        \end{tikzpicture}
      };
      \draw [->] (A0) -- node[above,font=\scriptsize] {$\textcolor{RedViolet}{45}\!\to\!\textcolor{Green}{457}$} (A1);
      \draw [->] (A1) -- node[above,font=\scriptsize] {$\textcolor{Green}{457}\!\to\!\textcolor{Green}{67}$} (A2);
      \draw [->] (A2) -- node[above,font=\scriptsize] {$\textcolor{Green}{67}\!\to\!\textcolor{RedViolet}{6}$} (A3);
    \end{tikzpicture}
  \end{center}
\end{example}

%% file: software.tex
\section{Software}\label{sec:software}
We have created a julia package \texttt{TropicalGalaxies.jl} in order to facilitate our experiments.  It relies on the computer algebra system \texttt{OSCAR} \cite{OSCAR,OSCAR-book}, and it is publicly available under the following url where installation and usage instructions can be found:
\begin{center}
  \url{https://github.com/YueRen/TropicalGalaxies.jl}.
\end{center}

It features functions for constructing and excising Laman graphs, see \cref{fig:softwareLaman} for the visualisations.  The Laman graphs are taken from the database in \cite{CapcoGalletGraseggerKoutschanLubbesSchicho2018a}.

\begin{jllisting}
G = laman_graph(4,1)
HH = excise(G, [2, 3])
FF = excise(HH, [1, 5])
visualize(G) # see Figure 9
visualize(HH)
visualize(FF)
\end{jllisting}

\begin{figure}[t]
  \centering
  \begin{tikzpicture}
    \node (left)
    {
      \includegraphics[width=0.25\linewidth]{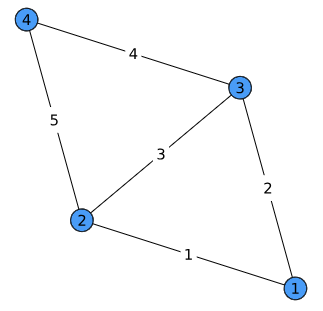}
    };
    \node[anchor=west,xshift=10mm] (middle) at (left.east)
    {
      \includegraphics[width=0.25\linewidth]{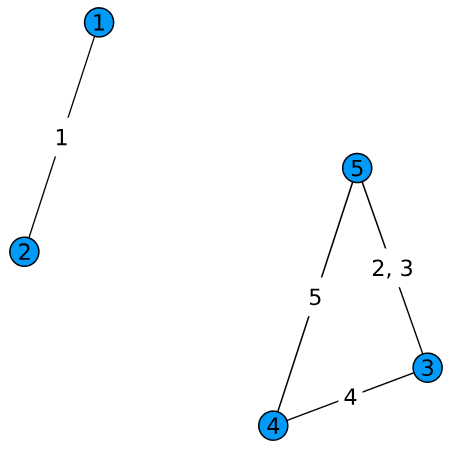}
    };
    \node[anchor=west,xshift=10mm] (right) at (middle.east)
    {
      \includegraphics[width=0.25\linewidth]{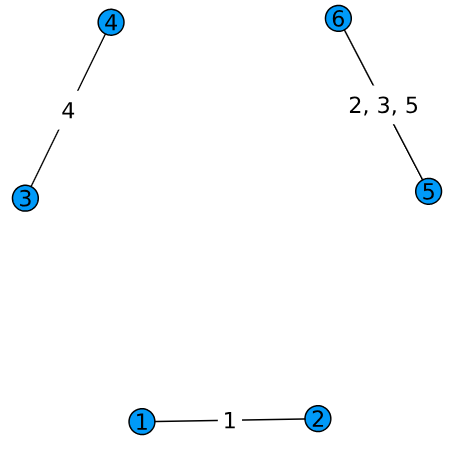}
    };
  \end{tikzpicture}
  \caption{A Laman graph, its excised triangle and its full excision\pointOrNoPoint}
  \label{fig:softwareLaman}
\end{figure}

It also allows for the construction of tropical galaxies, see \cref{fig:softwareGalaxy} for the visualization:

\begin{jllisting}
G = laman_graph(4,1)
Gamma = tropical_galaxy(G)
visualize_excision_graph(Gamma) # see Figure 10
\end{jllisting}

\begin{figure}[t]
  \centering
  \includegraphics[scale=0.8]{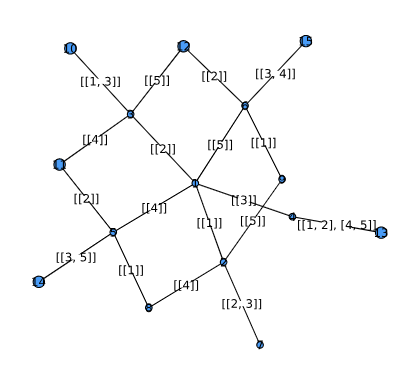}
  \caption{The tropical galaxy of the complete graph on $4$ vertices without an edge\pointOrNoPoint}
  \label{fig:softwareGalaxy}
\end{figure}

Up-to-date documentation can be found under the url above.
